\documentclass[12pt,reqno]{amsart}
\usepackage{amsmath,amsfonts,amsthm,amssymb,amscd}
\usepackage{graphicx,graphics}
\usepackage{psfrag}
\usepackage{cite}

\def\mydate{\number\year-\ifnum\month<10{0}\fi\number\month-\ifnum\day<10{0}\fi\number\day}

\newcommand{\Real}{{\mathbb{R}}}
\newcommand{\Cplx}{{\mathbb{C}}}

\newcommand{\ex}{\mathrm{e}}
\newcommand{\im}{\mathrm{i}}
\newcommand{\vfi}{\varphi}

\newcommand{\tssum}{\textstyle\sum}

\newcommand{\weakto}{\rightharpoonup}

\newcommand{\Disk}{\mathbb{D}}
\newcommand{\CMV}{{\mathcal{C}}}
\newcommand{\Err}{\mathcal{E}}
\newcommand{\Td}{\mathcal{T}}

\newcommand{\dmu}{\mathrm{d}\mu}
\newcommand{\tmu}{\tilde\mu}
\newcommand{\dtmu}{\mathrm{d}\tilde\mu}
\newcommand{\dphi}{\mathrm{d}\phi}

\newcommand{\ftil}{\widetilde{f}}

\newcommand{\hdim}{\mathrm{dim}_H}

\newcommand{\mom}[2]{c_{{#1}}^{(#2)}}

\let\det=\undefined\DeclareMathOperator*{\det}{det}

\newtheorem{theorem}{Theorem}[section]
\newtheorem{prop}[theorem]{Proposition}

\theoremstyle{definition}

\theoremstyle{remark}
\newtheorem*{remark}{Remark}

\begin{document}

\title[Spectral applications of McMullen's algorithm]{Some spectral applications of McMullen's Hausdorff dimension algorithm}

\author[K.~Gittins, N.~Peyerimhoff, M.~Stoiciu and D.~Wirosoetisno]{K.~Gittins, N.~Peyerimhoff, M.~Stoiciu and D.~Wirosoetisno}

\address{Katie Gittins\\
  Mathematical Sciences, Durham University\\
  Mount\-joy Site, South Road\\
  Durham\ \ DH1 3LE, UK}
\email{katie.gittins@durham.ac.uk}

\address{Norbert Peyerimhoff\\
  Mathematical Sciences, Durham University\\
  Mountjoy Site, South Road\\
  Durham\ \ DH1 3LE, UK}
\email{norbert.peyerimhoff@durham.ac.uk}

\address{Mihai Stoiciu\\
  Department of Mathematics and Statistics\\
  Williams College\\
  Williamstown, MA 01267, USA}
\email{mstoiciu@williams.edu}

\address{Djoko Wirosoetisno\\
  Mathematical Sciences, Durham University\\
  Mountjoy Site, South Road\\
  Durham\ \ DH1 3LE, UK}
\email{djoko.wirosoetisno@durham.ac.uk}

\date{\mydate}

\begin{abstract}
Using McMullen's Hausdorff dimension algorithm, we study numerically the
dimension of the limit set of groups generated by reflections along three
geodesics on the hyperbolic plane.
Varying these geodesics, we found four minima in the two-dimen\-sional
parameter space, leading to a rigorous result why this must be so.
Extending the algorithm to compute the limit measure and its moments,
we study orthogonal polynomials on the unit circle associated
with this measure.
Several numerical observations on certain coefficients related to these
moments and on the zeros of the polynomials are discussed.
\end{abstract}

\maketitle

% ===========================================================================

\section{Introduction}

Curtis McMullen introduced in \cite{mcmullen:98} a very efficient
algorithm for the computation of the Hausdorff dimension of the limit
set $\Lambda_\infty$ of general conformal dynamical systems.
Taking as our dynamical system a group $\Gamma = \Gamma_{g_0,\dots,g_k}$
generated by reflections along $k+1$ geodesics in the Poincar{\'e}
unit disk $\Disk$,
we study the limit set $\Lambda_\infty\subset S^1=\partial\Disk$
and the (unique) limit measure $\mu$ supported on it.
In this article, we study
two spectral aspects of the group, its limit set and its limit measure.

\medskip
The first spectral aspect arises from the intimate connection between
$\hdim(\Lambda_\infty)$ on the one hand, and, on the other,
the bottom of the spectrum of the Laplacian in the infinite-area
hyperbolic surface $S$,
defined as the double cover of the quotient $\Disk/\Gamma$
[\citen{borthwick:stiahs},\thinspace Ch.~14;\thinspace\citen{sullivan:87}].
In particular, we study the Hausdorff dimension $D(d_0,d_1,d_2)$ of a group
$\Gamma_{g_0,g_1,g_2}$ parameterised by three distances $d_0$, $d_1$ and $d_2$
with $d_0+d_1+d_2=d$.
Despite our initial expectation that the global minimum of $D$ is attained
only in the symmetric configuration $d_0=d_1=d_2$ (centre of triangle in
Fig.~\ref{f:triangle}),
numerical computations gave us three other points (open circles),
where this global minimum is attained.
This can be explained via a particular symmetry in the distance parameters,
which we subsequently proved rigorously in Prop.~\ref{t:sym}.
Moreover, our numerical results suggest that $D\ge1/2$ whenever $d_i=0$ for
some $i$ and that the global maxima of $D$ is attained when either
$d_i=d$ or $d_i=d_j=d/2$ (filled circles in Fig.~\ref{f:triangle}).
The interplay between the group $\Gamma_{g_0,g_1,g_2}$ and the associated
hyperbolic surface $S$ is crucial to arrive at the spectral results.
(See also \cite{baragar:06,phillips-sarnak:85} for further studies of
Laplace eigenvalues on families of hyperbolic surfaces.)

\begin{figure}[h]
\begin{center}
\psfrag{A}[c][c][0.707][0]{$A$}
\psfrag{B}[c][c][0.707][0]{$B$}
\psfrag{T}[c][c][0.707][0]{$\Td$}
\includegraphics[height=4cm]{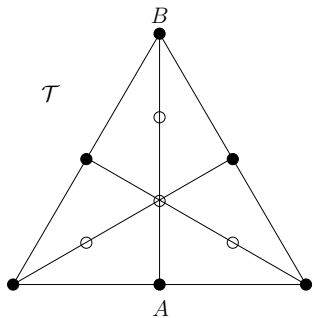}
\caption{
Four minima (open circles) and $6$ maxima (filled circles) of $D(d_0,d_1,d_2)$.
Barycentric coordinates subject to $d_0+d_1+d_2=d$;
see \S\ref{s:exsd} for details.
}\label{f:triangle}
\end{center}
\end{figure}

\medskip
The second spectral aspect arises in connection with
a family of unitary matrices called {\em CMV matrices\/} and
orthogonal polynomials on the unit circle ({OPUCs});
see \cite{simon:opuc1,simon:opuc2} for an encyclopaedic reference.
As discovered by Cantero, Moral and Vel{\'a}zquez
\cite{cantero-moral-velazquez:03},
the multiplication operator $f(z) \mapsto z f(z)$ is represented
in a suitable basis for $L^2(S^1;\mu)$
by a semi-infinite penta\-diagonal matrix bearing
their names which can be regarded as the unitary analogue of a
Jacobi matrix or a discrete Schr\"odinger operator.
% simon:05 (5.9):
The characteristic polynomials of the $k\times k$ simple truncation of this
CMV matrix is a polynomial $\Phi_k(z)$ that forms an orthogonal set in
$L^2(S^1;\mu)$.
As do their real counterparts, $\Phi_k(z)$ satisfy a recurrence
relation \eqref{q:szego}, whose complex coefficients $\alpha_k$ are
known as {\em Verblunsky's coefficients}.

One example considered by McMullen in \cite{mcmullen:98} is the case
of three symmetric geodesics $g_0,g_1,g_2$ in $\Disk$,
dubbed ``symmetric pair of pants''.
As the opening angle $\theta$ of each geodesic varies from $0$ to $2\pi/3$,
the limit measure $\mu_\theta$ describes a continuous transition from an
atomic measure supported at three points to the absolutely continuous
Lebesgue measure on $S^1$.
Extending McMullen's algorithm, we computed (approximations to) the
measure $\mu_\theta$ and its associated moments.
These are then used to study the zeros of certain ``paraorthogonal''
polynomials $\Phi_k(z;\beta)$, which are the $k\times k$ {\em unitary\/}
truncations of the CMV matrix depending on a parameter $\beta$.
Among our numerical observations, we found that
the Verblunsky coefficients are all negative and are monotonic
in $\theta$ [cf.\ Fig.~\ref{f:verb}].
We also observed that the zeros of the paraorthogonal $\Phi_k(z;\beta)$
tend to cluster together near the gaps of supp$\,\mu_\theta$ and
they are in some sense monotone in $\beta$ [cf.\ Fig.~\ref{f:beta}].

\medskip
Finally, we note that a different algorithm for the calculation of
the Hausdorff dimension was introduced in \cite{jenkinson-pollicott:02}.
Baragar \cite{baragar:06} also described an algorithm for the calculation
of the Hausdorff dimension of three geodesics $g_0,g_1,g_2$ with two
distances equal to zero,
for which McMullen's algorithm is not applicable.

\medskip
The rest of this paper is structured as follows.
For the reader's convenience, we first recall basic facts about hyperbolic
reflection groups in \S\ref{s:limsetlimmeas} and describe McMullen's
symmetric example in \S\ref{s:SPP}.
Descriptions of McMullen's Hausdorff dimension algorithm and
its extension follow in Section~\ref{s:alg}.
In Section~\ref{s:eigvalsurf}, we prove a symmetry property of the
Hausdorff dimension $D(d_0,d_1,d_2)$ and present our numerical computation
of $D$.
Section~\ref{s:opucmv} describes our numerical observations related to
OPUCs and CMV matrices, after a brief review of the known theory.

\subsection*{Acknowledgements}
We would like to thank Patrick Dorey, John Parker, Mark Pollicott
and Scott Thompson for useful comments and suggestions.
NP is grateful for the hospitality of Williams College.
KG was supported by a Nuffield Undergraduate Research Bursary.

% \vfill\eject

% \bigskip\hbox to\hsize{\qquad\hrulefill\qquad}\medskip

% ===========================================================================

\section{Hyperbolic reflection groups}

\subsection{Limit sets and limit measures}
\label{s:limsetlimmeas}

Let $\Disk = \{ z \in \Cplx : |z| < 1 \}$ denote the Poincar{\'e} unit disk
with its hyperbolic distance function $d(\cdot,\cdot)$ \cite[\S7.2]{beardon:gdg}.
Geodesics are circular Euclidean arcs, meeting the
boundary $S^1 = \partial \Disk$ perpendicularly.
Let $C_g$ denote the Euclidean circle representing the geodesic $g$,
i.e.\ $g = \Disk \cap C_g$.
Then the hyperbolic reflection in $g$ agrees with the restriction
(to the unit disk) of the Euclidean reflection in the circle $C_g$.

Let $g_0,\ldots,g_k$ be $k+1$ geodesics with corresponding Euclidean
circles $C_0,\ldots,C_k$ and closed disks $D_0,\ldots,D_k \subset \Cplx$
such that $C_l = \partial D_l$.
We assume that the geodesics are not nested, i.e.\ the disks
$D_0,\ldots,D_k$ are pairwise disjoint.
To this setting, we associate the discrete group
$\Gamma = \Gamma_{g_0,\ldots,g_k}$ generated by the reflections $\rho_l$
in the geodesics $g_l$. $\Gamma$ acts on $\Disk$ by hyperbolic isometries.
This action extends to the boundary $S^1$, and the limit
set of $\Gamma$ is given by $\Lambda_\infty(\Gamma)=\overline{\Gamma
  \cdot p} \cap S^1$ (which is independent of $p \in \Disk$). If we
have at least three geodesics, the limit set $\Lambda_\infty(\Gamma)
\subset S^1$ has a Cantor-like structure. This can be seen as follows:
We refer to the disks $D_l^{(0)} = D_l$ as the {\em primitive cells}
or {\em cells of generation zero}. Let $\rho_l$ be the Euclidean
reflection in $C_l$. The {\em cells of generation one} are given by
$\rho_s(D_l^{(0)})$ with $s \neq l$. Every primitive cell $D_l^{(0)}$
contains $k$ cells of generation one. The {\em cells of generation
  two} are obtained by reflecting all cells of generation one in all
primitive circles in which they are not contained. Repeating this
operation, we obtain successive generations of cells and a
nested structure, where every cell of generation $n \ge 1$ has a
unique parent (of generation $n-1$) and precisely $k$ children (of
generation $n+1$).
If $A_n$ denotes the union of all cells of generation $n$,
we recover $\Lambda_\infty(\Gamma) \subset S^1$ as
the intersection $\bigcap_{n \ge 1} A_n$.

A $\Gamma$-invariant {\em conformal density of dimension $\delta \ge
  0$} is a measure $\mu$ supported on $\Lambda_\infty(\Gamma)$
satisfying, for all $\gamma \in \Gamma$ and all continuous functions
$f: S^1 \to \Cplx$,
\begin{equation}\label{eq:mutransform}
  \int_{S^1} f(z)\, \dmu
	= \int_{S^1} f \circ \gamma(z) \, |\gamma'(z)|^\delta\, \dmu,
\end{equation}
where $\gamma'$ denotes the ordinary derivative of $\gamma: S^1 \to
S^1$ with respect to the angle metric in radians. This means that
$\mu$ behaves like a $\delta$-dimensional Hausdorff measure.

Many properties of conformal densities can be derived from the
following explicit construction, due to Patterson \cite{patterson:76} and
Sullivan \cite{sullivan:79,sullivan:84}.
Let $\Gamma_0 \subset \Gamma$ be the index-2 subgroup of
orientation-preserving isometries (compositions of an even number
of reflections) and
let $g_s(x,y) = \sum_{\gamma\in\Gamma_0} \ex^{-sd(x,\gamma y)}$.
Then there exists a critical exponent $\delta(\Gamma)$,
independent of $x,y \in \Disk$, such that $g_s(x,y)$ is convergent
for $s > \delta(\Gamma)$ and divergent for $s<\delta(\Gamma)$.
Since $\Gamma_0$ is geometrically finite, $g_s(x,y)$ diverges also at
$s = \delta(\Gamma)$ \cite[Thm.~9.31]{nicholls:91}.
Let
\begin{equation}\label{eq:patsulrep}
   \mu_s = \frac{1}{g_s(y,y)}\sum_{\gamma\in\Gamma_0}
		\ex^{-s d(x,\gamma y)} \delta_{\gamma y},
\end{equation}
where $\delta_z$ denotes the Dirac measure at $z$.
By Helly's Theorem, there are weak limits of $\mu_{s_j}$ for certain sequences
$s_j \searrow \delta(\Gamma)$ and, due to the divergence at the critical
exponent, each such weak limit is supported on the boundary $S^1$.
For the following results we refer to \cite[Ch.~9]{nicholls:91} or to
the original articles by Patterson and Sullivan:

\begin{theorem}\label{thm:confdensprops}
Let $\Gamma_0$ be the index-two subgroup of $\Gamma =
\Gamma_{g_1,\ldots,g_k}$, $\mu_s$ as defined in \eqref{eq:patsulrep},
and $\delta(\Gamma)$ be the critical exponent. Then
\begin{itemize}
\item[(i)] every weak limit of $\mu_s$ is supported on
    $\Lambda_\infty(\Gamma)$;
\item[(ii)] every weak limit at $x=0$ satisfies the transformation
    rule \eqref{eq:mutransform} and is therefore a
    $\Gamma$-invariant conformal density;
\item[(iii)] for $x=y=0$, all weak limits are probability measures;
\item[(iv)] there is only one conformal density $\mu$ of dimension
    $\delta(\Gamma)$, up to scaling;
\item[(v)] the conformal density $\mu$ of dimension $\delta(\Gamma)$
    has no point masses;
\item[(vi)] the critical exponent coincides with the Hausdorff
    dimension of $\Lambda_\infty(\Gamma)$.
\end{itemize}
\end{theorem}

Henceforth, we refer to the unique conformal probability density
$\mu_\Gamma$ of dimension $\delta(\Gamma)$ as the {\em limit measure} of the
hyperbolic reflection group $\Gamma$.

% ---------------------------------------------------------------------------

\begin{figure}[ha]
\begin{center}
\psfrag{g0}[l][l][0.707][0]{$D^{(0)}_0$}
\psfrag{g1}[l][l][0.707][0]{$D^{(0)}_1$}
\psfrag{g2}[l][l][0.707][0]{$D^{(0)}_2$}
\psfrag{g10}[l][l][0.707][0]{$D^{(1)}_0$}
\psfrag{g11}[l][l][0.707][0]{$D^{(1)}_1$}
\psfrag{g12}[l][l][0.707][0]{$D^{(1)}_2$}
\psfrag{g13}[l][l][0.707][0]{$D^{(1)}_3$}
\psfrag{g14}[l][l][0.707][0]{$D^{(1)}_4$}
\psfrag{g15}[l][l][0.707][0]{$D^{(1)}_5$}
\includegraphics[scale=0.1]{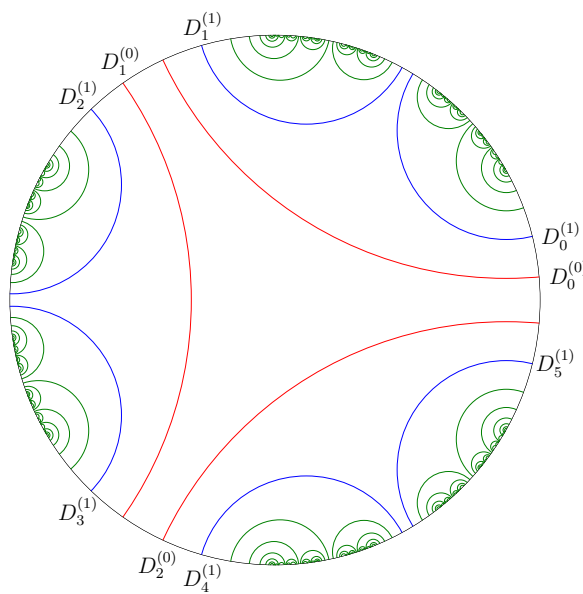}
\caption{
The geodesics for SPP with $\theta=110^\circ$.
The primitive [red in online version] geodesics are of the same size,
as are generation-1 [blue] ones, but not higher-gen\-eration [green] ones.
Note the 6-fold symmetry.
}\label{f:spp}
\end{center}
\end{figure}

\subsection{McMullen's ``Symmetric Pair of Pants''}
\label{s:SPP}

We now introduce the ``Symmetric Pairs of Pants'' from \cite{mcmullen:98}.
This example, henceforth SPP, will be important in
Section~\ref{s:opucmv} below.
For $\theta\in(0,2\pi/3)$,
let $g_0,g_1,g_2$ be three symmetrically placed geodesics
with end points $z_j\, \ex^{\pm\im\theta/2}$ for $j=0,1,2$,
where $z_j = \zeta_6^{2j+1}$ and $\zeta_6 = \ex^{2\pi\im/6}$
[the largest/red geodesics in Fig.~\ref{f:spp}].
For brevity, we denote the corresponding reflection group by
$\Gamma_\theta$, and the corresponding limit set and limit measure
by $\Lambda_\infty(\theta)$ and $\mu_\theta$, respectively.
It is clear from the construction that $\Gamma_\theta$ respects the 6-fold
symmetry generated by $z\mapsto\bar z$ and $z\mapsto\ex^{\pm2\pi\im/3}z$.
A graph of the corresponding Hausdorff dimension
$\theta \mapsto \delta(\Gamma_\theta) = \hdim(\Lambda_\infty(\theta))$
can be found in \cite[Fig.~3]{mcmullen:98}.
The family of limit measures $\{ \mu_\theta\}_{0<\theta<2\pi/3}$ represents
a continuous transition (via singular continuous measures) from the
purely atomic measure
$\mu_0 := \frac{1}{3}(\delta_{z_0}+\delta_{z_1}+\delta_{z_2})$
to the Lebesgue probability measure $\mu_{2\pi/3}$
with $\dmu_{2\pi/3} := \dphi/(2\pi)$:

\begin{prop}
  Let $\theta_n, \theta \in [0,2\pi/3]$ and $\theta_n \to
  \theta$. Then the limit measures $\mu_{\theta_n}$ converge weakly to
  $\mu_\theta$.
\end{prop}

\begin{proof}
  The weak convergence in the case $\theta = 0$ follows from the
  symmetry of the configuration and the fact that the primitive cells
  $D_0^{(0)}$, $D_1^{(0)}$, $D_2^{(0)}$ shrink to the points $z_0,z_1,z_2$,
  as $\theta_n \to 0$. For $\theta \neq 0$, the proof of
  \cite[Thm. 3.1]{mcmullen:98} implies not only the continuity of the
  Hausdorff dimension, but also the weak convergence of the limit
  measures \cite[Thm.~1.4]{mcmullen:99}.  For $\theta = 2\pi/3$,
  $\Gamma_0$ is a finitely generated Fuchsian group of the first
  kind. As noted at the end of \cite{patterson:76}, the corresponding limit
  measure agrees with the normalized Lebesgue measure.
\end{proof}

% ===========================================================================

\section{McMullen's algorithm}\label{s:alg}

McMullen's eigenvalue algorithm \cite{mcmullen:98} provides a very
effective numerical method to compute the Hausdorff dimension of very
general conformal dynamical systems.
For the reader's convenience, we briefly recall this algorithm for
the special case of a hyperbolic reflection group
$\Gamma = \Gamma_{g_0,\ldots,g_k}$.
The algorithm can be extended to obtain arbitrarily good weak approximations
of the limit measure $\mu_\Gamma$ by atomic measures $\mu_\Gamma^\epsilon$,
allowing us to obtain highly accurate values for the moments of $\mu_\Gamma$.
This high precision is necessary to perform the numerical spectral analysis
of the CMV matrices in Section~\ref{s:opucmv}.

% ---------------------------------------------------------------------------

\subsection{The algorithm}\label{s:alg-ctm}

Let $\epsilon > 0$ be given.
% We use the notions introduced in the previous section.
We start with the primitive cells $D_l^{(0)}$
associated to the geodesics $g_l$ and the corresponding Euclidean
reflections $\rho_l$ in the circles $C_l$.
Next, we generate all cells of generation one,
$\{D_0^{(1)},D_1^{(1)},\ldots,D_{k(k+1)}^{(1)}\}$,
via $\rho_s(D_l^{(0)})$ with $s \neq l$
[the 6 second-largest/blue geodesics in Fig.~\ref{f:spp}].
At step $n+1$, we start with a finite list of cells
$\{D_0^{(n)},D_1^{(n)},\ldots\}$
and build up a new list of cells $\{D_0^{(n+1)},D_1^{(n+1)},\ldots\}$
by the following procedure.
If the radius of $D^{(n)}_j$ is less than $\epsilon$, we put it back
into the list and relabel it $D^{(n+1)}_{j'}$;
otherwise we replace it by its $k$ children
$D^{(n+1)}_{j'},\ldots,D^{(n+1)}_{j'+k-1}$
(i.e.\ the $k$ cells of generation $n+1$ contained in $D^{(n)}_j$).
Since the radius of every non-primitive cell is at most $\eta<1$ times
the radius of its parent, we eventually obtain a list
$\{D^{(N)}_0,\ldots,D^{(N)}_M\}$ of cells of generations {\em at most\/} $N$
and of radii less than $\epsilon$.
We note that every cell in this final list has a parent of radius
not less than $\epsilon$.

Let $z_0,\ldots,z_M \in S^1$ be the radial projections of the centres
of the Euclidean disks $D_0^{(N)},\ldots,D_M^{(N)}$ to the unit
circle. McMullen introduces the notation $i \mapsto j$ if $D_j^{(N)}
\subset \rho_k(D_i^{(N)})$, where $k$ is the index of the primitive
cell containing $D_i^{(N)}$, and defines the sparse $(M+1)\times(M+1)$
matrix $T$ with the entries
\begin{equation}
   T_{ij} = \begin{cases}
	|\rho_k'(\rho_k(z_i))|^{-1}	&\textrm{if } i \mapsto j,\\
	0				&\textrm{otherwise}.
\end{cases}
\end{equation}
Let $T^d$ denote the matrix obtained by raising each entry of $T$ to the
power $d \in \Real$, $(T^d)_{ij} = T_{ij}^d$, and $\lambda(A)$ denote
the spectral radius of the matrix $A$.
The underlying dynamical system implies that $T^d$ is a primitive
non-negative matrix, and we can apply the Perron--Frobenius theorem.
Note that $\lambda(T^d)$ can be effectively computed via the power method.

We note that the construction (and variables) above depend on the
choice of $\epsilon > 0$.
To emphasize this dependency, we sometimes write, e.g.,
$z_l^{(\epsilon)}$ and $M(\epsilon)$ instead of $z_l$ and $M$.
Let $d_\epsilon$ be the unique positive number such that
$\lambda(T^{d_\epsilon}) = 1$; note that our $d_\epsilon$ is denoted
$\alpha_n$ in McMullen's article, which here denotes the Verblunsky
coefficients in Section~\ref{s:opucmv}.
Then \cite[Thm.~2.2]{mcmullen:98}
\begin{equation}\label{q:dcty}
   \lim_{\epsilon \to 0} d_\epsilon =  \hdim(\Lambda_\infty(\Gamma)).
\end{equation}

It is clear from the construction that the sample points $z_l$ and
their weights $w_l$ respect the 6-fold symmetry of $\Gamma_\theta$,
viz., if $z_l$ is a sample point with weight (i.e.\ corresponding
entry in the Perron--Frobenius eigenvector) $w_l$, then so are
$\overline{z_l}$, $\ex^{\pm2\pi\im/3}z_l$ and $\ex^{\pm2\pi\im/3}\overline{z_l}$.

% ---------------------------------------------------------------------------

\subsection{Approximations of the limit measures}
\label{s:approxlimmeas}

Let $w_0^{(\epsilon)},\ldots,w_M^{(\epsilon)}$, with $M = M(\epsilon)$,
be the entries of the Perron--Frobenius eigenvector of
$T^{d_\epsilon}$, normalized so that $\sum_l w_l^{(\epsilon)}=1$.
All entries are positive and can be considered as approximations
of the values $\mu_\Gamma(S^1 \cap D_l^{(N)})$.
In fact, we have

\begin{prop}\label{t:muepsconvmu}
For every $\epsilon >0$, let
\begin{equation}\label{q:atomu}
   \mu_\Gamma^{(\epsilon)}
	= \sum_{l=0}^{M(\epsilon)} w_l^{(\epsilon)} \delta_{z_l^{(\epsilon)}}
\end{equation}
be the atomic measure supported on the points $z_l^{(\epsilon)}$
(obtained by McMullen's algorithm) with weights $w_l^{(\epsilon)}$.
Then the probability measures $\mu_\Gamma^{(\epsilon)}$ converge weakly
to the limit measure $\mu_\Gamma$ as $\epsilon \to 0$.
\end{prop}

\begin{proof}
Writing $z_l=z_l^{(\epsilon)}$ and $w_l=w_l^{(\epsilon)}$ for conciseness,
first note that for $z_s$ in $D_j^{(0)}$, $j \in \{0,1,2\}$,
the corresponding component of the Perron--Frobenius eigenvector satisfies
\begin{equation}\label{q:wsrel}
    w_s = \sum_{l:\, l \mapsto s} |\rho_j'(z_l)|^{d_\epsilon} w_l
\end{equation}
since $|\rho_j'(z)| = |\rho_j'(\rho_j(z))|^{-1}$.
By the uniqueness of the limit measure [Thm.~\ref{thm:confdensprops}(iv)],
we only need to show that every weak limit of $\mu_\Gamma^{(\epsilon)}$
satisfies the transformation property \eqref{eq:mutransform}.
It suffices to prove this for the generators $\rho_l$.
We discuss the case $\gamma = \rho_0$.

Let $\delta > 0$ and $g \in C(S^1)$ be given.
Assume first that the intersection ${\rm supp}\, g \cap D_l^{(0)}$
is non-trivial only for $l = 0$.
Let $d_0 = \hdim(\Lambda_\infty(\Gamma))$.
Since the set $A_1 = \bigcup_l D_l^{(1)}$ has a positive Euclidean distance to
both the centre and the boundary of the primitive cell $D_0^{(0)}$,
there exists a $C_{{\rm inv}} > 0$ such that
$1/C_{\rm inv} \le |\rho_j'(z)| \le C_{{\rm inv}}$
for all $z \in A_1$.
Using \eqref{q:dcty}, there exists $\epsilon_0 > 0$ such that,
for all $z \in A_1$ and all $\epsilon < \epsilon_0$,
\begin{equation}\label{eq:rhopowdiff}
  \bigl| |\rho_0'(z)|^{d_\epsilon} - |\rho_0'(z)|^{d_0} \bigr|
	< \frac{\delta}{2\,\|g\|_\infty}.
\end{equation}
Since $g$ is uniformly continuous, there exists $\epsilon_1>0$,   such that
\begin{equation}
   |g(z) - g(z')| < \frac{\delta}{2} \qquad
  \forall\, |z-z'| < 2 \epsilon_1.
\end{equation}
Let $\epsilon < \min \{\epsilon_0,\epsilon_1\}$ be fixed.
We conclude from \eqref{eq:rhopowdiff} that
\begin{equation}\label{q:g00}\begin{aligned}
   \biggl| \int_{S^1} g&(\rho_0(z)) |\rho_0'(z)|^{d_0} \,
      	\dmu_\Gamma^{(\epsilon)} - \int_{S^1} g(z) \,
      	\dmu_\Gamma^{(\epsilon)} \biggr|\\
   &< \frac{\delta}{2} + \biggl|
	\sum_{l=0}^M g(\rho_0(z_l)) |\rho_0'(z_l)|^{d_\epsilon} w_l -
	\int_{S^1} g(z) \, \dmu_\Gamma^{(\epsilon)} \biggr|.
\end{aligned}\end{equation}
Let us assume that all cells $D_l^{(N)}$ contained in the primitive
cell $D_0^{(0)}$ are numbered from $D_0^{(N)}$ to $D_{M_0}^{(N)}$;
note that $N$ and $M_0$ both depend on $\epsilon$.
Since $g(\rho_0(z_l))\ne0$ only if $l \mapsto j$ for
some $j \in \{0,\ldots,M_0\}$,
because only then we have $\rho_0(z_l) \in D_j^{(N)} \subset D_0^{(0)}$,
we have
\begin{equation}
   \sum_{l=0}^M\, g(\rho_0(z_l))\,|\rho_0'(z_l)|^{d_\epsilon}\,w_l
	= \sum_{j=0}^{M_0}\sum_{l:\,l\mapsto j}\, g(\rho_0(z_l))\,|\rho_0'(z_l)|^{d_\epsilon}\,w_l\,.
\end{equation}
Using the fact that $\mu_\Gamma^{(\epsilon)}$ is an atomic measure and
using \eqref{q:wsrel}, we have
\begin{equation}
   \int_{S^1} g(z) \dmu_\Gamma^{(\epsilon)}
	= \sum_{j=0}^{M_0}\sum_{l:\,l\mapsto j}\, g(z_j)\,|\rho_0'(z_l)|^{d_\epsilon} w_l\,.
\end{equation}
Putting these together and using \eqref{q:wsrel} again for the first inequality,
\begin{equation}\label{q:g01}\begin{aligned}
    &\left| \sum_{l=0}^M\, g(\rho_0(z_l)) |\rho_0'(z_l)|^{d_\epsilon} w_l -
      \int_{S^1} g(z) \, \dmu_\Gamma^{(\epsilon)} \right|\\
    &\quad {}= \left| \sum_{j=0}^{M_0} \sum_{l:\, l \mapsto j} \bigl\{
        g(\rho_0(z_l))-g(z_j) \bigr\} |\rho_0'(z_l)|^{d_\epsilon} w_l \right| \le
    \frac{\delta}{2} \sum_{j=0}^{M_0} w_j \le \frac{\delta}{2}.
\end{aligned}\end{equation}
Combining \eqref{q:g00} and \eqref{q:g01}, we obtain
\begin{equation}\label{eq:transformapprox}
  \left| \int_{S^1} g\circ \rho_0(z) |\rho_0'(z)|^{d_0}\,
    \dmu_\Gamma^{(\epsilon)} - \int_{S^1} g(z)\,
    \dmu_\Gamma^{(\epsilon)} \right| < \delta.
\end{equation}
Now assume that ${\rm supp}\, g \cap D_0^{(0)} = \emptyset$.
Then we obtain the same estimate \eqref{eq:transformapprox}, by applying the
above arguments to $h(z) = g\circ \rho_0(z) |\rho_0'(z)|^{d_0}$ and
using $|\rho_0'(z)| = |\rho_0'(\rho_0(z))|^{-1}$.
The general case is obtained by choosing a partition of unity
$\chi_1,\chi_2 \in C(S^1)$ with supports disjoint to
$\bigcup_{l \ge 1} D_l^{(0)}$ and $D_0^{(0)}$, respectively.
\end{proof}

% ---------------------------------------------------------------------------

\subsection{Moment computations for SPP}
\label{s:momcomp}

Given a probability measure $\mu$ on $S^1$, we have the
Hilbert space $L^2(S^1;\mu)$ with inner product
\begin{equation}\label{q:inner}
   \langle f, g \rangle := \int_{S^1} \overline{f(z)} g(z) \;\dmu.
\end{equation}
Now let $\theta\in(0,2\pi/3)$ be fixed and consider the limit measure
$\mu_\theta$ of SPP constructed above.
Associated to $\mu_\theta^{}$ and $L^2(S^1;\mu_\theta)$
are its {\em moments\/}
\begin{equation}\label{q:cdef}
   c_k^{(\theta)} = \langle z^k,1 \rangle = \int_{S^1} z^{-k}\;\dmu_\theta.
\end{equation}
The symmetries of the underlying classical system imply that
$c_k^{(\theta)} = c_{-k}^{(\theta)} \in\Real$,
and that only every third moment is non-zero [cf.~\eqref{q:c3} below].
These moments encapsulate much of the information contained in $\mu_\theta^{}$,
e.g., knowledge of them is sufficient to ``quantize'' the classical dynamics
of hyperbolic reflections to the unitary CMV matrices of
Section~\ref{s:opucmv}.

We obtain approximations for the moments by computing the
moments of the atomic measure [cf.~\eqref{q:atomu}]
\begin{equation}
   \mu_\theta^{(\epsilon)}
	= \sum_{l=0}^{M(\epsilon)}\, w_l^{(\epsilon)} \delta_{z_l^{(\epsilon)}}
\end{equation}
for small enough $\epsilon > 0$.
With $\zeta_6=\ex^{2\pi\im/6}$ and recalling the 6-fold symmetry of
the points $z_l$, we rewrite the points $z_l^{(\epsilon)}$
as $\zeta_6^{2j+1}e^{\pm i \phi_l^{(\epsilon)}}$ with $j \in \{0,1,2\}$,
$0 \le l \le M/6=M(\epsilon)/6$ and
\begin{equation}
   0 < \phi_0^{(\epsilon)} < \ldots < \phi_{M/6}^{(\epsilon)} < \theta/2.
\end{equation}
For every fixed $l$, the Perron--Frobenius entries corresponding to the
six points $\zeta_6^{2j+1}e^{\pm i \phi_l^{(\epsilon)}}$ agree, and we
denote their value by $m_l^{(\epsilon)} > 0$.
Then the following holds.

\begin{prop}\label{t:moments}
The moments $c_s^{(\theta)}$ of the limit measure $\mu_\theta$ of SPP satisfy
\begin{equation}\label{q:c3}
   c_{3j-1}^{(\theta)} = c_{3j+1}^{(\theta)} = 0
\end{equation}
and
\begin{equation}
   c_{3j}^{(\theta)} = (-1)^j\, \lim_{\epsilon \to 0}\,
	\sum_{l=0}^{M/6}\, 6 m_l^{(\epsilon)} \cos\bigl[3j\phi_l^{(\epsilon)}\bigr].
\end{equation}
\end{prop}

\begin{proof}
Let $c_s^{(\theta,\epsilon)}$ denote the $s$th moment of the measure
$\mu_\theta^{(\epsilon)}$.
For brevity, we write $m_l=m_l^{(\epsilon)}$ and $\phi_l=\phi_l^{(\epsilon)}$.
The 6-fold symmetry yields
\begin{align}
   c_s^{(\theta,\epsilon)}
	&= \sum_{l=0}^{M/6} 2 m_l \bigl\{ \cos[(\pi/3+\phi_l)s] + \cos[(\pi/3-\phi_l)s] +  (-1)^s \cos(s \phi_l) \bigr\} \notag\\
%    &= \sum\nolimits_l 2 m_l \bigl\{ 2 \cos(\pi s/3) + (-1)^s) \cos(\phi_l s)\bigr\} \notag\\
    &= \begin{cases} 0 & \text{if $s \equiv \pm 1 \mod 3$,}\\
      (-1)^k \sum_l 6 m_l \cos(s \phi_l) & \text{if $s \equiv 0 \mod 3$.}
    \end{cases}
\end{align}
By Proposition \ref{t:muepsconvmu}, we have
$\mu_\theta^{(\epsilon)} \weakto \mu_\theta$, finishing the proof.
\end{proof}

\begin{figure}[h]
\begin{center}
\includegraphics[scale=0.45]{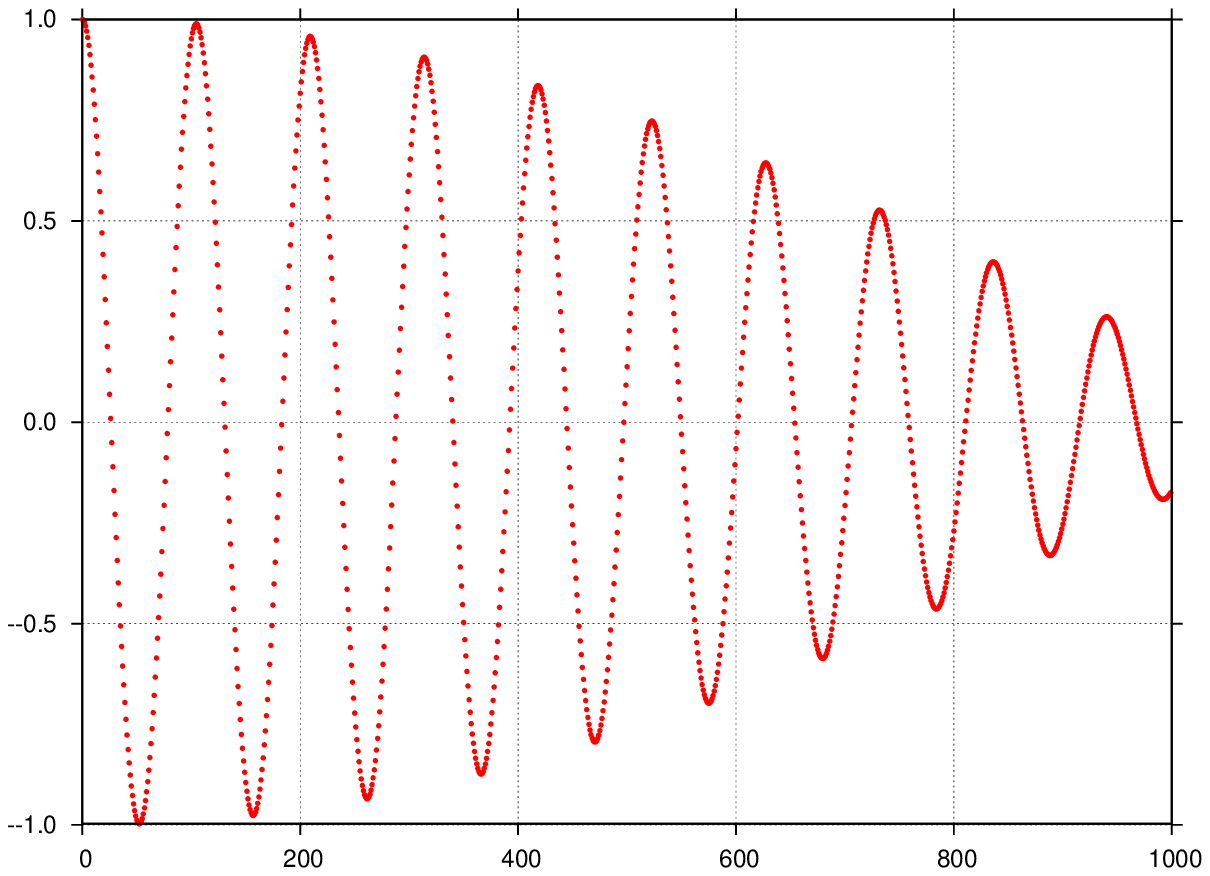}
\includegraphics[scale=0.45]{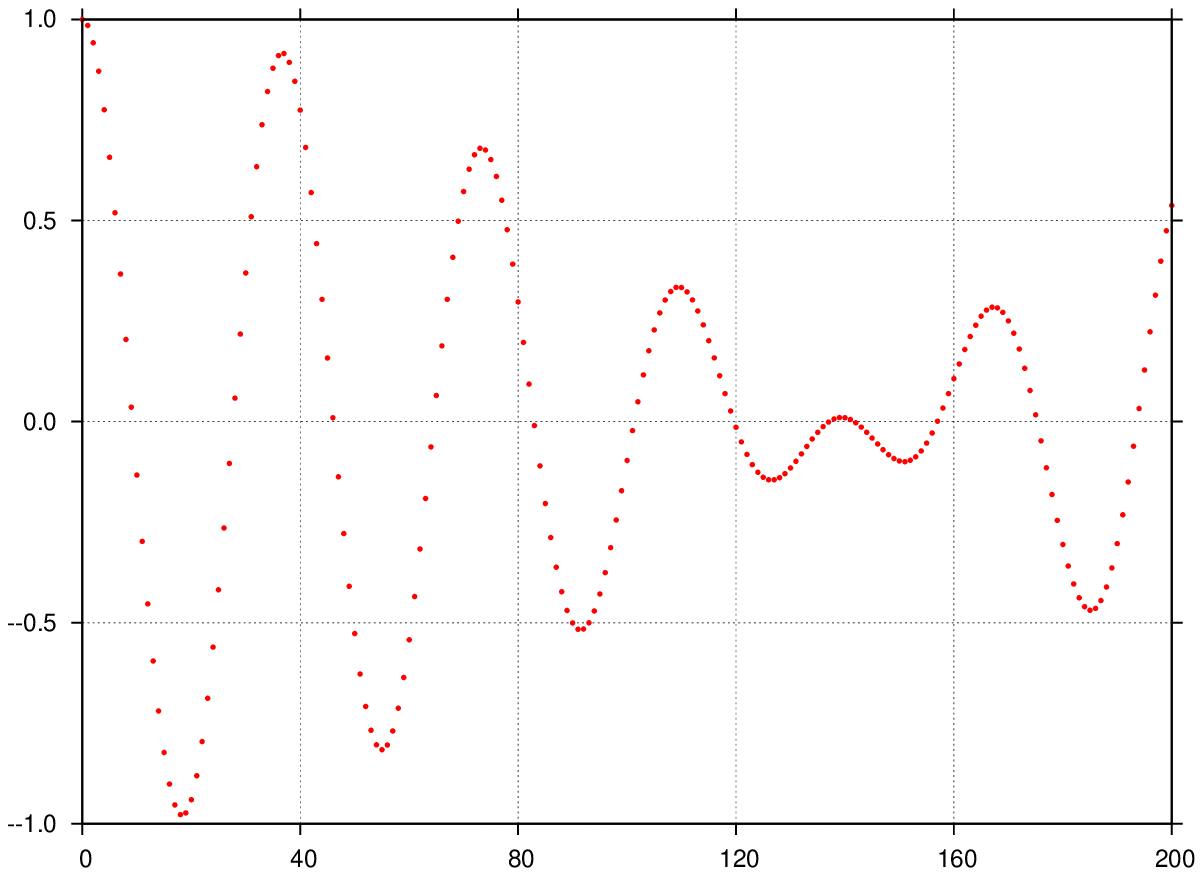}
\includegraphics[scale=0.45]{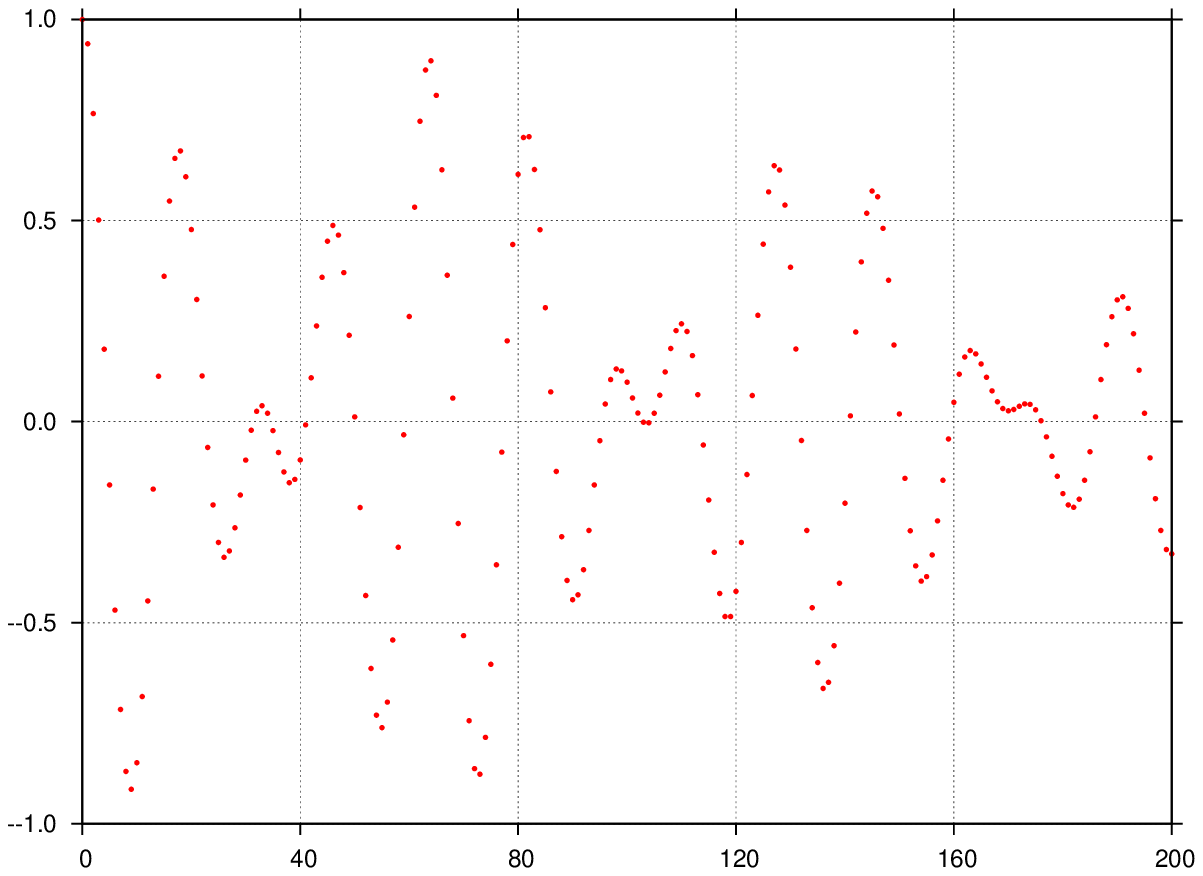}
\includegraphics[scale=0.45]{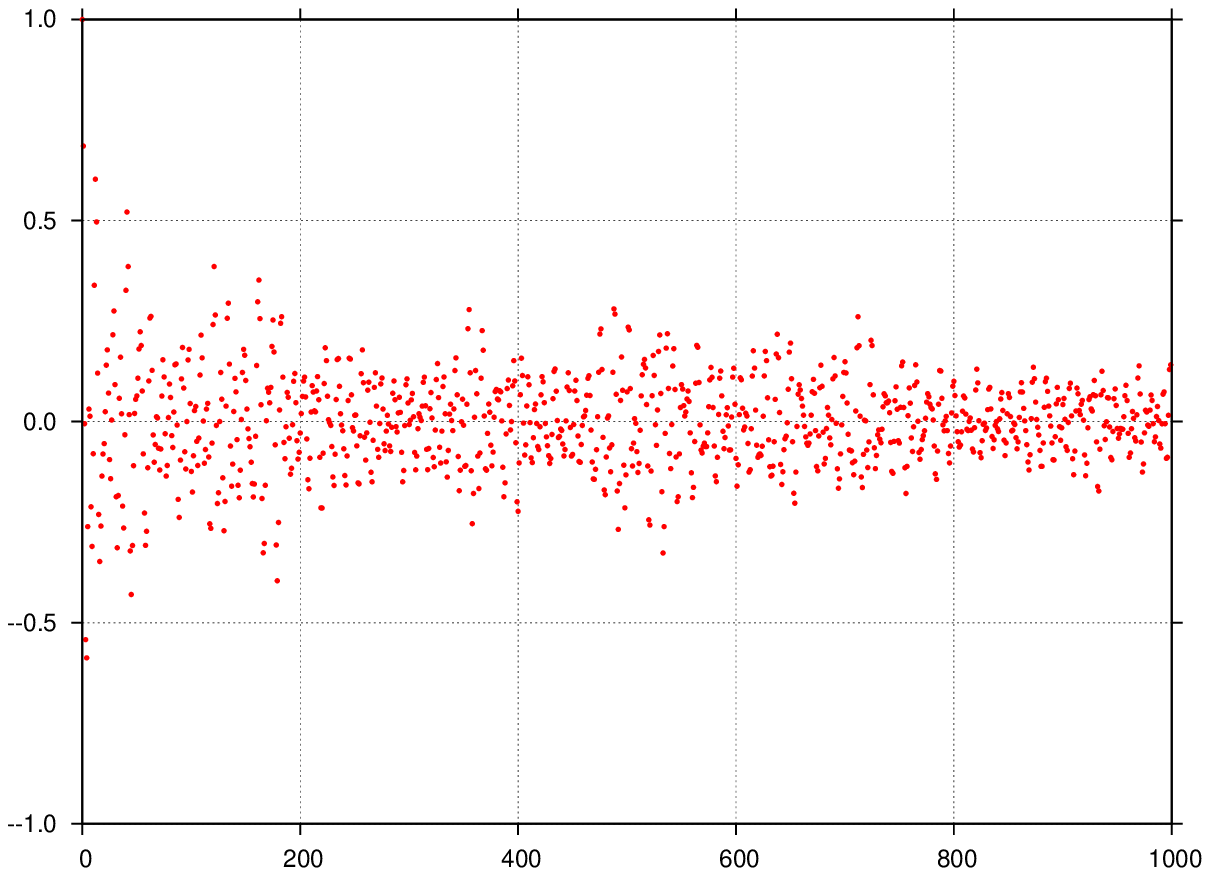}
\caption{Numerically computed non-zero moments $(-1)^k\mom{3k}{\theta}$
for $\theta=30^\circ$, $50^\circ$ (top row),
$70^\circ$ and $100^\circ$ (bottom row).
The horizontal axis is $k$ and $\epsilon=10^{-10}$.}\label{f:mom}
\end{center}
\end{figure}

Figure~\ref{f:mom} shows the numerically computed
\cite{gittins-npi-stoiciu-dw:opuc} moments for limit
measures $\mu_\theta$ of different opening angles $\theta$.
For clarity, we actually plotted $(-1)^k c_{3k}^{(\theta)}$,
corresponding to the moments of $\mu_\theta$ rotated by $\pi$
\cite[(1.6.63)]{simon:opuc1}.
For small $\theta$, the moments $c_{3j}^{(\theta)}$ appear to oscillate
with a higher frequency with an amplitude modulated by another lower frequency.
This is not unexpected since, for small $\theta$, the cells of the second
generation are tiny.
Since all points $z_l$ fall into these cells, the positions of the second
generation cells control the general oscillation behaviour of the moments.
As $\theta$ increases, the cells
of every generation grow and the distances between successive cells of
a given generation decreases, with higher-generation cells
contributing different frequencies to the graph of the moments.
The graph seems to become increasingly erratic.

% ===========================================================================

\section{Least eigenvalue of hyperbolic surfaces}
\label{s:eigvalsurf}

% ---------------------------------------------------------------------------

\subsection{Spectrum of $S_{d_0,d_1,d_2}$ and symmetries}

Given three lengths $d_0$, $d_1$, $d_2 > 0$,
there is (up to isometries) a unique genus-0 hyperbolic surface
$S_{d_0,d_1,d_2}^{}$ of infinite area with three cylindrical ends
(funnels) whose short geodesics are of lengths $2d_0$, $2d_1$ and $2d_2$
[Fig.~\ref{f:surfhex}a].
Referring to Fig.~\ref{f:surfhex}b,
this surface is the oriented double cover of $\Disk / \Gamma_{g_0,g_1,g_2}$,
where $d_j$ is the hyperbolic distance between two geodesics.
The three geodesics $g_j$, together with the geodesics realising the
distances $d_j$, define a unique (up to isometries) hyperbolic
right-angled hexagon.
The three heights of this hexagon are concurrent
\cite[Thm.~2.4.3]{buser:gscrs}, and we choose their
intersection point to be the origin of $\Disk$.
The angle $\varphi$ and the hyperbolic length $t$ in Fig.~\ref{f:surfhex}b
can be obtained using trigonometric identities for trirectangles
and pentagons \cite[Ch.~2]{buser:gscrs}, giving
\begin{equation}\begin{aligned}
  \cos \varphi &= \frac{\delta_0^2+\delta_0 \delta_1 \delta_2}
  {\sqrt{\delta_0^2+\delta_1^2+2\delta_0\delta_1\delta_2}
    \sqrt{\delta_0^2+\delta_2^2+2\delta_0\delta_1\delta_2}},\\
  \tanh t &= \delta_0 \delta_2
  \sqrt{\frac{\delta_0^2+\delta_1^2+\delta_2^2+2\delta_0\delta_1\delta_2-1}
    {(\delta_1^2+\delta_0^2\delta_2^2+2\delta_0\delta_1\delta_2)
      (\delta_0^2+\delta_2^2+2\delta_0\delta_1\delta_2)}},
\end{aligned}\end{equation}
where $\delta_j = \cosh d_j$.
The Euclidean distance from the origin of $\Disk$ is $x=\tanh(t/2)$.
These results allow us to compute explicitly, for given distances
$d_0,d_1,d_2$, the three primitive cells needed for McMullen's algorithm.

\begin{figure}[h]
\begin{center}
\psfrag{g0}[c][c][0.707][0]{$g_0$}
\psfrag{g1}[c][c][0.707][0]{$g_1$}
\psfrag{g2}[c][c][0.707][0]{$g_2$}
\psfrag{d0}[c][c][0.707][0]{$d_0$}
\psfrag{2d0}[c][c][0.707][0]{$2 d_0$}
\psfrag{d1}[c][c][0.707][0]{$d_1$}
\psfrag{2d1}[c][c][0.707][0]{$2 d_1$}
\psfrag{d2}[c][c][0.707][0]{$d_2$}
\psfrag{2d2}[c][c][0.707][0]{$2 d_2$}
\psfrag{phi}[c][c][0.707][0]{$\varphi$}
\psfrag{t}[c][c][0.707][0]{$t$}
\psfrag{Sd0d1d2}[c][c][0.707][0]{$S_{d_0,d_1,d_2}$}
\psfrag{(a)}{(a)}
\psfrag{(b)}{(b)}
\includegraphics[scale=0.3]{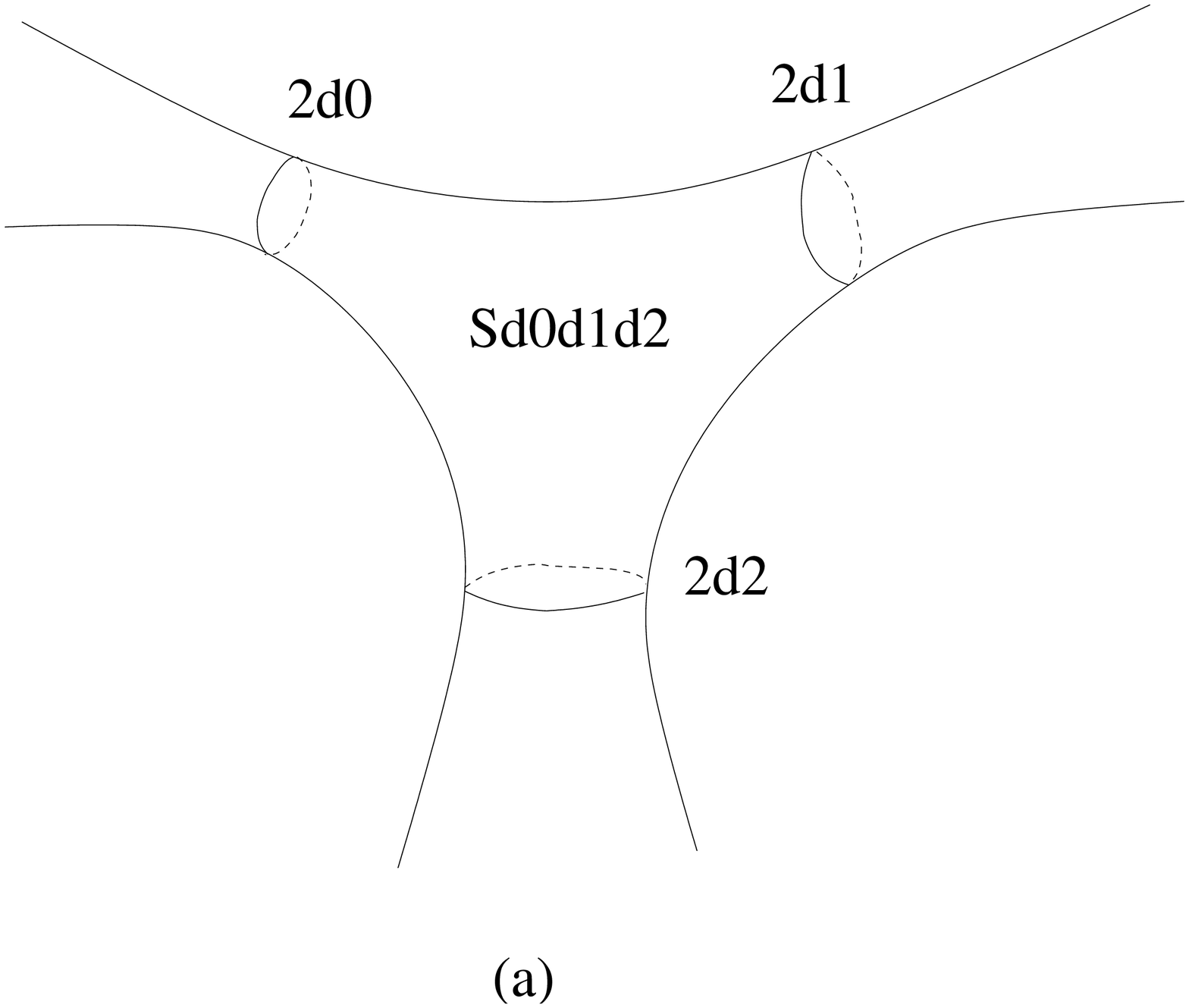}
\includegraphics[scale=0.3]{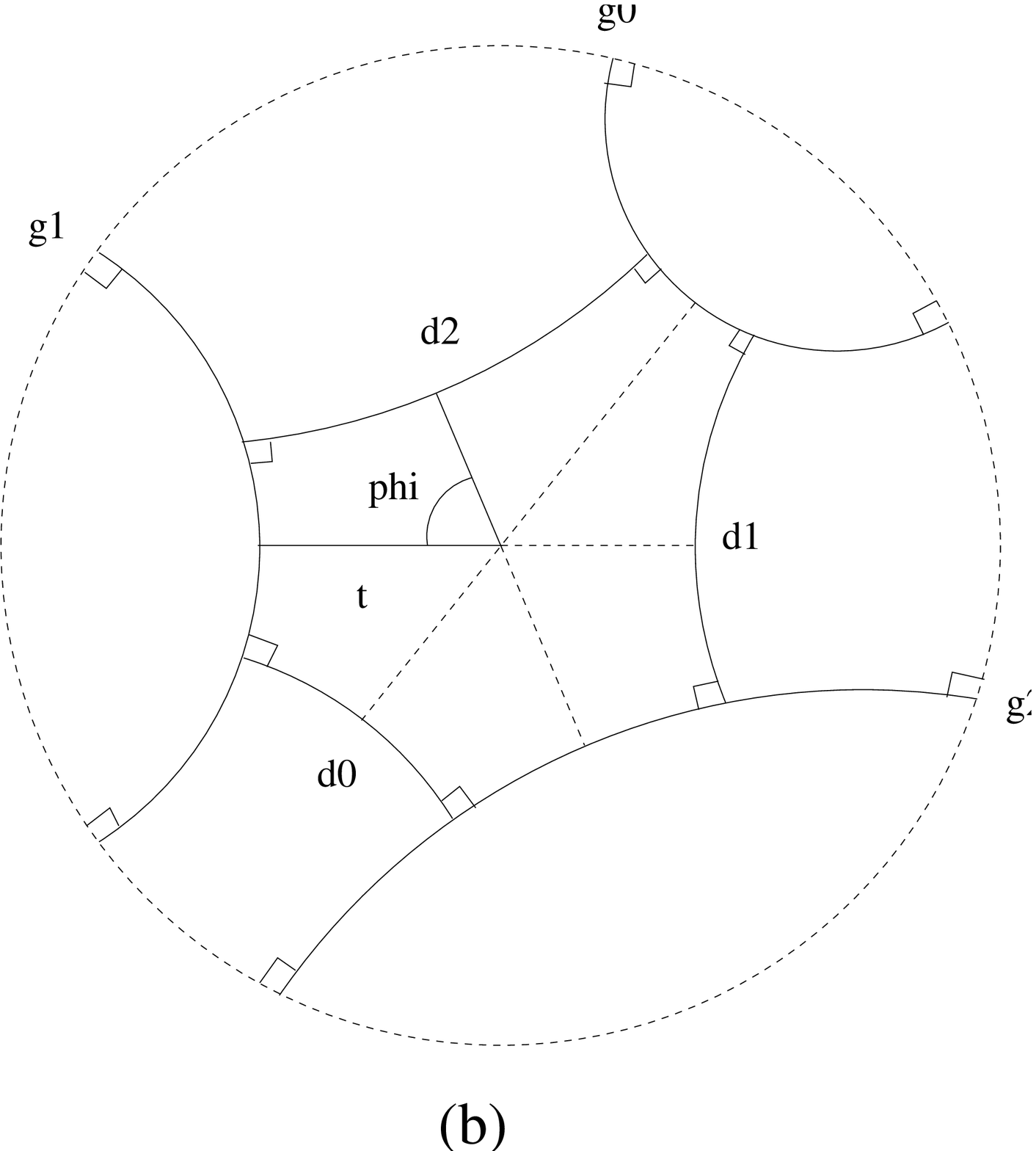}
\caption{(a)~The hyperbolic surface $S_{d_0,d_1,d_2}$.
(b)~The corresponding hexagon with concurrent heights.}\label{f:surfhex}
\end{center}
\end{figure}

Let $D(d_0,d_1,d_2) := \Lambda_\infty(\Gamma_{g_0,g_1,g_2})$ denote
the Hausdorff dimension of the limit set, and
let $\lambda_0(S_{d_0,d_1,d_2})$ be the bottom of the spectrum of the
(positive) Laplace operator on the surface $S_{d_0,d_1,d_2}$.
The following result is well known and relates the above two functions.
(The result holds generally for geometrically finite hyperbolic
manifolds of infinite volume.)

\begin{theorem}\cite[Thm.~2.21]{sullivan:87}\label{thm:hdiml0}
We have, with the notions from above,
\begin{equation}
\lambda_0(S_{d_0,d_1,d_2}) =
  \begin{cases} \frac{1}{4} &\text{if } D = D(d_0,d_1,d_2) \le \frac{1}{2},\\
	\> D(1-D) &\text{if } D = D(d_0,d_1,d_2) \ge \frac{1}{2},
\end{cases}
\end{equation}
and $\lambda_0(S_{l_0,l_1,l_2}) < 1/4$ is the eigenvalue of a
positive $L^2$-eigenfunction if and only if $D(d_0,d_1,d_2) > 1/2$.
\end{theorem}

The spectral interpretation of the Hausdorff dimension was extended by
Patterson in \cite{patterson:88} (see also \cite[Thm.~14.15]{borthwick:stiahs})
to the case $D=D(d_0,d_1,d_2) < 1/2$, in which case $D$ represents
the location of the first resonance.

Fixing a total distance $d > 0$, we consider the function
$(d_0,d_1,d_2) \mapsto D(d_0,d_1,d_2)$ with $d_0+d_1+d_2 = d$.
One might think that the point $p_0 = (d/3,d/3,d/3)$ would be a
{\em unique} global minimum of the Hausdorff dimension subject to this
total distance restriction (in fact, this was our initial guess).
Numerical computations showed, however, that the function
$D(d_0,d_1,d_2)$ is more complicated and has, in fact, precisely
{\em four\/} global minima at $p_0$ and at $3$ other points,
\begin{equation*}
  \Bigl( \frac{2d}{3},\frac{d}{6},\frac{d}{6} \Bigr),\;
  \Bigl( \frac{d}{6},\frac{2d}{3},\frac{d}{6} \Bigr)
  \quad\textrm{and}\quad
  \Bigl( \frac{d}{6},\frac{d}{6},\frac{2d}{3} \Bigr).
\end{equation*}
This observation lead to the following rigorous result:

\begin{prop}\label{t:sym}
Let $a,b > 0$.
Then we have
\begin{equation}\label{q:dsym}
  D(2b,a,a) = D(2a,b,b).
\end{equation}
\end{prop}

We note that Proposition \ref{t:sym} shows that, for $d > 0$, the
function $t \mapsto D(t,\frac{d-t}{2},\frac{d-t}{2})$ is symmetric
with respect to $d/2$.
In Fig.~\ref{f:triangle}, this means that $D$ is symmetric about
the midpoint along the line $AB$.

\begin{proof}
The hexagon with the three geodesics $g_0,g_1,g_2$ with distances $a,a,2b$
is shown in Fig.~\ref{f:funddoms}b.
The configuration is obviously symmetric with respect to the horizontal
geodesic $g_3$.
Let $\rho_j$ denote the hyperbolic reflection in $g_j$ and $D_j$ be the
closed disk with $g_j = \partial D_j$.
Then $\Gamma_{g_0,g_1,g_2} = \langle \rho_0,\rho_1,\rho_2 \rangle=:\Gamma_0$,
and a fundamental domain of $\Gamma_0$ is given by
\begin{equation}
   {\mathcal F}_0 = \{ z \in \Cplx : \mathrm{Re}(z) > 0 \}
	\backslash (D_0 \cup D_2).
\end{equation}
The geodesic $g_4$ in Fig.~\ref{f:funddoms}b is the reflection of
$g_2$ in $g_1$.
Therefore, we have for the corresponding reflection
$\rho_4 = \rho_1 \rho_2 \rho_1$.
A fundamental domain of
$\widehat \Gamma = \langle \rho_0,\rho_1,\rho_2,\rho_3,\rho_4 \rangle$
is
\begin{equation}
   \widehat {\mathcal F} = \{ z \in \Cplx : \mathrm{Re}(z) > 0,\;
\mathrm{Im}(z) < 0 \} \backslash D_2.
\end{equation}
Since ${\mathcal F}_0= \widehat {\mathcal F}\cup \rho_3(\widehat {\mathcal F})$
we have $[\widehat \Gamma: \Gamma_0 ] = 2$.
The geodesics $g_2,g_3,g_4$ in Fig.~\ref{f:funddoms}b have distances $b,b,2a$.
Completely analogously, noting that
$\Gamma_{g_2,g_3,g_4} = \langle \rho_2,\rho_3,\rho_4 \rangle =: \Gamma_1$
has
\begin{equation}
   {\mathcal F}_1 = \{ z \in \Cplx : \mathrm{Im}(z) < 0 \}
   \backslash (D_2 \cup D_4)
\end{equation}
as a fundamental domain, we have $\rho_0 = \rho_3 \rho_2 \rho_3$, and
${\mathcal F}_1= \widehat {\mathcal F} \cup \rho_1(\widehat {\mathcal F})$.
This implies, again, $[\widehat \Gamma: \Gamma_1 ] = 2$ and, therefore,
\begin{equation}
  \hdim \Lambda_\infty(\Gamma_0) = \hdim \Lambda_\infty(\Gamma_1)
	= \hdim \Lambda_\infty(\widehat \Gamma).
\end{equation}

\begin{figure}[h]
\begin{center}
\psfrag{D0}[c][c][0.707][0]{$D_0$}
\psfrag{D1}[c][c][0.707][0]{$D_1$}
\psfrag{D2}[c][c][0.707][0]{$D_2$}
\psfrag{D3}[c][c][0.707][0]{$D_3$}
\psfrag{D4}[c][c][0.707][0]{$D_4$}
\psfrag{D5}[c][c][0.707][0]{$D_5$}
\psfrag{g0}[c][c][0.707][0]{$g_0$}
\psfrag{g1}[c][c][0.707][0]{$g_1$}
\psfrag{g2}[c][c][0.707][0]{$g_2$}
\psfrag{g3}[c][c][0.707][0]{$g_3$}
\psfrag{g4}[c][c][0.707][0]{$g_4$}
\psfrag{g5}[c][c][0.707][0]{$g_5$}
\psfrag{a}[c][c][0.707][0]{$a$}
\psfrag{b}[c][c][0.707][0]{$b$}
\psfrag{2a}[c][c][0.707][0]{$2a$}
\psfrag{2b}[c][c][0.707][0]{$2b$}
\psfrag{4a}[c][c][0.707][0]{$4a$}
\psfrag{4b}[c][c][0.707][0]{$4b$}
\psfrag{(a)}{(a)}
\psfrag{(b)}{(b)}
\psfrag{(c)}{(c)}
\psfrag{alpha}[c][c][0.707][0]{$\alpha$}
\psfrag{beta}[c][c][0.707][0]{$\beta$}
\includegraphics[width=\textwidth]{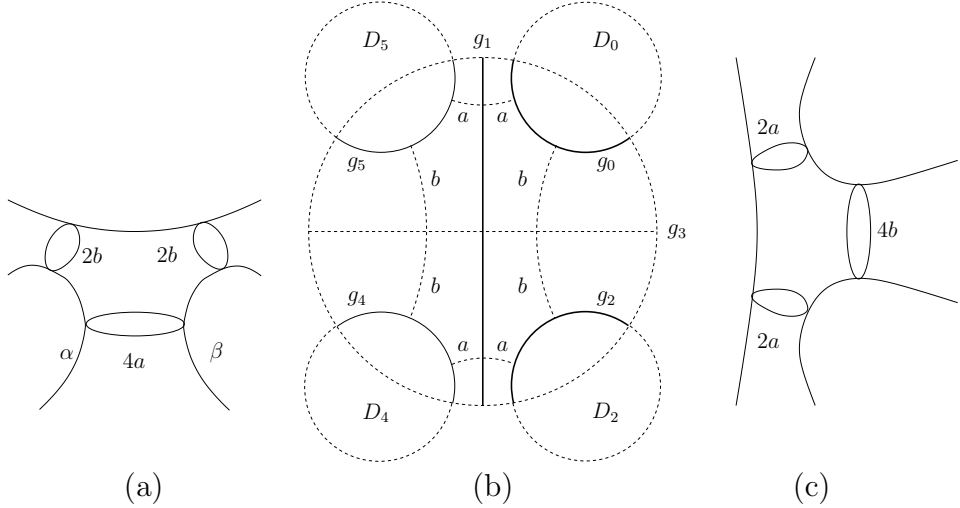}
\caption{Surfaces: (a)~$S_{2a,b,b}$ and (c)~$S_{2b,a,a}$.
(b)~The associated domains in $\Disk$.}\label{f:funddoms}
\end{center}
\end{figure}
\end{proof}

\medskip

\begin{remark}
The spectrum of a geometrically finite hyperbolic surface of
infinite area consists of an absolutely continuous part
$[1/4,\infty)$ without embedded eigenvalues and finitely many
eigenvalues of finite multiplicity in the interval $(0,1/4)$;
see \cite{lax-phillips:82,lax-phillips:84a,lax-phillips:84b,lax-phillips:85}
or \cite[Ch.~7]{borthwick:stiahs}.

Assuming $D(2a,b,b) > 1/2$, we conclude from Thm.~\ref{thm:hdiml0}
and Prop.~\ref{t:sym} that
$\lambda_0(S_{2a,b,b}) = \lambda_0(S_{2b,a,a})$, i.e.\ $S_{2a,b,b}$
and $S_{2b,a,a}$ have the same lowest eigenvalue.
It is natural to ask whether these two surfaces are iso\-spectral.
We obtain $S_{2b,a,a}$ from $S_{2a,b,b}$ by
first cutting $S_{2a,b,b}$ along the geodesics
$\alpha$, $\beta$ in Fig.~\ref{f:funddoms}a, unfolding it onto
$\Disk\backslash(D_0\cup D_2\cup D_4\cup D_5)\allowbreak=:D_S$
in Fig.~\ref{f:funddoms}b,
and then gluing the boundary geodesics $g_0,g_5$ and $g_2,g_4$ together.
Let $\rho_j$ be the reflection along geodesic $g_j$ as before,
with $\rho_1,\rho_3$ also considered to act isometrically on $S_{2a,b,b}$
and $S_{2b,a,a}$.
Since $\rho_1$, $\rho_3$ and the Laplacian $\Delta$ commute, we consider
simultaneous eigenfunctions of these operators on $S_{2a,b,b}$.
Let $f$ be an $L^2$-eigenfunction in $S_{2a,b,b}$ with eigenvalue $\lambda<1/4$
which is even under $\rho_1$ and $\rho_3$,
\begin{equation}\label{q:efcn}
   \Delta f=\lambda f
   \qquad
   \textrm{with}\quad f\circ\rho_1 = f = f\circ\rho_3.
\end{equation}
(At least one such $f$ exists---with $\lambda=\lambda_0$.)
Now any such $f\in C^\infty(S_{2a,b,b})$ can be transplanted to
$\ftil\in C^\infty(D_S)$.
Thanks to the symmetry (\ref{q:efcn}b), $\ftil$ along with all its
derivatives agree along the corresponding points on $g_0,g_5$ and
$g_2,g_4$, so $\ftil$ can in turn be transplanted to an $L^2$-eigenfunction
$\hat f\in C^\infty(S_{2b,a,a})$ with the same eigenvalue,
$\Delta\hat f=\lambda\hat f$ in $S_{2b,a,a}$.
One can carry out the same argument when $f\circ\rho_1=-f=f\circ\rho_3$
in \eqref{q:efcn}, assuming such an $f$ exists.
This argument shows that some eigenvalues of these two surfaces coincide.
\end{remark}

% ---------------------------------------------------------------------------

\subsection{Example $S_{d_0,d_1,d_2}$ with $d_0+d_1+d_2=3$}\label{s:exsd}

Identifying the triple $(d_0,d_1,d_2)$ with the point
$(d_0+d_1\ex^{2\pi\im/3} + d_2\ex^{4\pi\im/3})\,\im/\sqrt3\in \Real^2$,
we represent the domain of $D(d_0,d_1,d_2)$ by an
equilateral triangle $\Td\subset \Real^2$ centred at $0$
with heights $\sqrt{3}/2$;
see Fig.~\ref{f:triangle}.
In Fig.~\ref{fig:hdimgraph} we plot $D(d_0,d_1,d_2)$ restricted
to $d_0+d_1+d_2 = 3$, computed using McMullen's algorithm.
When any $d_j$ is small (i.e.\ for points close to $\partial\Td$),
% thick black line in the plot),
it is difficult to compute $D$ numerically
using McMullen's algorithm even with $256$-bit arithmetic;
we have therefore left out the blank regions near $\partial\Td$
(thick black triangle on the plot).

\begin{figure}[h]
\begin{center}
\includegraphics[scale=1.2]{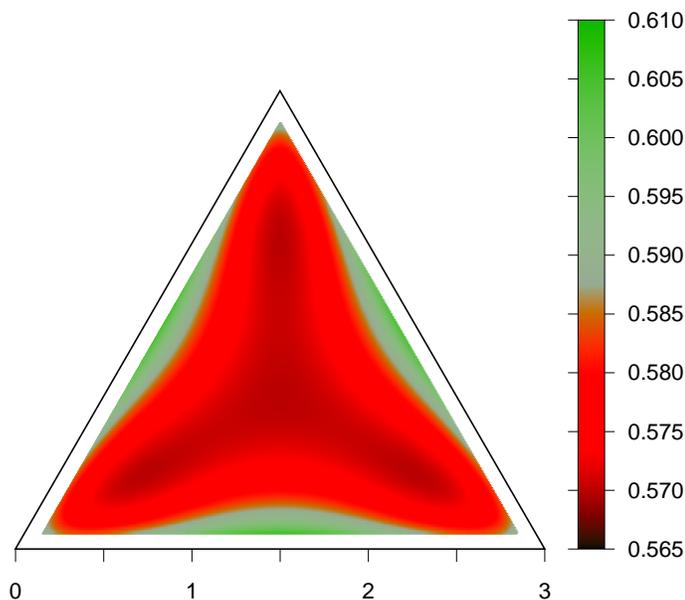}
\caption{Graph of the function $D(d_0,d_1,d_2)$ restricted to $d_0+d_1+d_2=3$;
bottom axis is $d_1$ with $d_0=0$.}\label{fig:hdimgraph}
\end{center}
\end{figure}

We found the minimum of $D$ to be $D(1,1,1) \doteq 0.56996\ldots > 1/2$,
which as noted above is also attained at three other points
$(2,1/2,1/2)$, $\ldots$.
We remark that $D$ is very flat for much of $\Td$ (red region in the plot),
only increasing rapidly near $\partial\Td$ (where our numerical computation
breaks down).
The four global minima of $D$ can also be seen in Fig.~\ref{f:hdimb},
which shows the bottom of $D$ at an expanded vertical scale.
In view of the apparent smoothness of the function $D$ in
Fig.~\ref{f:hdimb}, it is natural to ask whether there are
explicit formulas for the gradient or the Hessian of $D$.
From the shape of the graph (the part reliably computed), we believe that
$D$ attains its global maximum at the vertices and the midpoints of
$\partial\Td$.

\begin{figure}[h]
\begin{center}
\includegraphics[width=300pt]{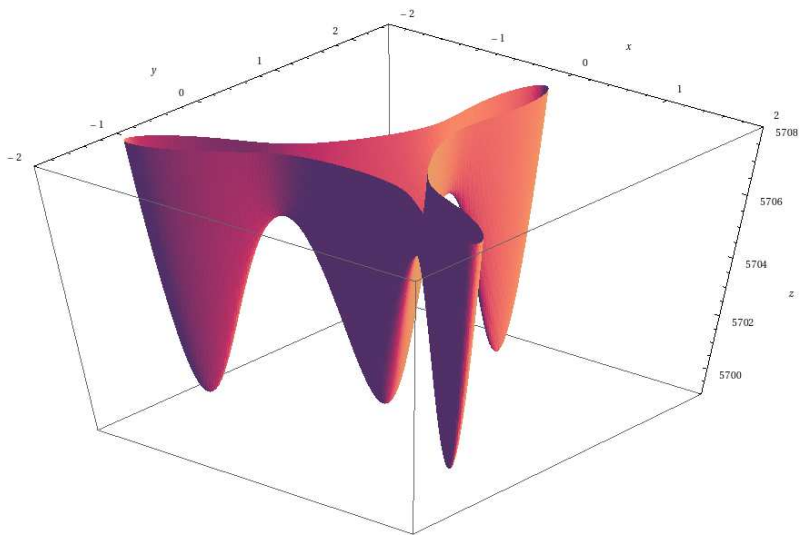}
\caption{The bottom part of $D(d_0,d_1,d_2)$; note scale.}\label{f:hdimb}
\end{center}
\end{figure}

The value of $D$ at these points can be computed using \eqref{q:dsym}
and the results of Baragar \cite{baragar:06}, who described an algorithm
for the Hausdorff dimension of reflection groups $\Gamma_{g_0,g_1,g_2}$
with two distances equal to zero.
Our $(d_0,d_1,d_2) = (3,0,0)$ corresponds to
$a=(1+x)/(1-x)\doteq 2.35240\ldots$ in \cite{baragar:06}
with $x = \tanh^2(3/4)$.
Using \cite[Table 1]{baragar:06}, we obtain $0.225 <
\lambda_0(S_{3,0,0}) < 0.226$ and
consequently $0.654 < D(3,0,0)=D(0,3/2,3/2) < 0.659$,
which we believe bounds the global maximum of $D$ in $\Td$.

\begin{remark}
Our numerical computations indicate that the situation does not
change qualitatively when we choose a much larger total distance $d$
(we also carried out detailed calculations for the case $d=6$
with the global minimum $0.334541\ldots$, and obtained graphs
very similar to Figs.\ \ref{fig:hdimgraph} and~\ref{f:hdimb}).
For any $\epsilon\in(0,\epsilon_0)$, there exists a total distance
$d=d(\epsilon)$ such the four global minima of $D$ under the restriction
$d_0+d_1+d_2=d$ agree with $\epsilon$.
Again, for a large part of $\Td$, the graph of $D$ should be relatively flat.

It seems plausible and there is strong numerical evidence (although
we do not have a proof) that one has
\begin{equation}\label{eq:Hdimmon}
D(d_0,d_1,d_2) \le D(d_0',d_1',d_2')
	\quad\textrm{if } d_j' \le d_j
	\textrm{ for } j=0,1,2.
\end{equation}
Assuming continuity of $D$ up to the boundary $\partial \Td$,
the graph must become very steep close to the boundary, because
McMullen's asymptotic $\lim_{c \to \infty} D(c,c,0) = 1/2$
\cite[Thm.~3.6]{mcmullen:98},
together with the monotonicity property \eqref{eq:Hdimmon},
imply that $D$ assumes values $\ge 1/2$ at all boundary points $\partial \Td$:
For $(a,b,0) \in \partial \Td$ with $a > b > 0$ we have
\begin{equation}
  D(a,b,0) \ge D(a,a,0) \ge \lim_{c \to \infty} D(c,c,0) = 1/2.
\end{equation}
\end{remark}

% ===========================================================================

\section{OPUC and CMV matrices}\label{s:opucmv}

% As alluded to in the Introduction, ...

% Aim: Study properties of POPUC limit measures of SPP.

% ---------------------------------------------------------------------------

\subsection{Overview}\label{s:opucmv-rev}

We review briefly the properties of OPUCs and CMV matrices as relevant
to the current work, referring the reader to
\cite{simon:opuc1,simon:opuc2,simon:05,simon:07} for more details.
As before, $\Disk\subset\Cplx$ denotes the open unit disk and
$S^1=\partial\Disk$.
Given a probability measure $\mu$ on $S^1$, supported on an infinite set, 
and the inner product \eqref{q:inner},
let $\{\Phi_k(z;\mu)\}_{k=0}^\infty$ be a set of monic polynomials orthogonal
with respect to $\mu$ (with the usual convention that
$\Phi_k(z;\mu)=z^k+{}$lower-order terms);
for brevity, we will often write $\Phi_k(z)$ for $\Phi_k(z;\mu)$ when there
is no confusion.
As do real orthogonal polynomials, $\{\Phi_k\}$ satisfy a recurrence relation
\begin{equation}\label{q:szego}
   \Phi_{k+1}(z) = z\Phi_k(z) - \bar\alpha_k\Phi_k^*(z)
\end{equation}
called {\em Szeg{\H o}'s recursion\/}, where the {\em Verblunsky coefficients\/}
$\alpha_k$ can be shown to lie in $\Disk$.
The {\em reversed polynomial\/} is
\begin{equation}
   \Phi_k^*(z) := z^k\overline{\Phi_k(1/\bar z)}
\end{equation}
or, with $\Phi_k(z)=\sum_{j=0}^k\,b_j z^j$,
\begin{equation}
   \Phi_k^*(z) = \tssum_{j=0}^k\, \bar b_{k-j} z^j.
\end{equation}
This implies that $\Phi_k^*(0)=1$ for all $k$ and,
together with \eqref{q:szego},
\begin{equation}\label{q:verb0}
   \alpha_k = -\overline{\Phi_{k+1}(0)}.
\end{equation}
In practice, one can compute $\{\Phi_k\}_{k=0}^\infty$ using Gram--Schmidt
on $\{z^k\}_{k=0}^\infty$.
We note that $\{\Phi_k\}$ may or may not form a basis for $L^2(S^1;\mu)$;
see \cite[Thm.~2.2]{simon:05}.

If, on the other hand,
we apply Gram--Schmidt to $\{1,z,z^{-1},z^2, z^{-2}$, $\ldots\}$,
we get orthonormal polynomials $\{\chi_0(z), \chi_1(z), \chi_2(z), \ldots\}$
which do form a basis for $L^2(S^1;\mu)$.
The {\em CMV matrix\/} associated to the measure $\mu$ is the matrix
representation of the operator $f(z) \to z f(z)$ on $L^2(S^1;\mu)$.
It has the semi-infinite pentadiagonal form
\begin{equation}
\CMV = \left(
\begin{array}{cccccc}
  \bar{\alpha}_0 & \bar{\alpha}_1 \rho_0 & \rho_1 \rho_0 & 0 & 0 & \ldots \\
  \rho_0 & -\bar{\alpha}_1 \alpha_0 & -\rho_1 \alpha_0 & 0 & 0 & \ldots\\
  0 & \bar{\alpha}_2 \rho_1 & -\bar{\alpha}_2 \alpha_1 &
  \bar{\alpha}_3 \rho_2 & \rho_3 \rho_2 & \ldots\\
  0 & \rho_2 \rho_1 & -\rho_2 \alpha_1 & -\bar{\alpha}_3 \alpha_2 &
  -\rho_3 \alpha_2 & \ldots\\
  0 & 0 & 0 & \bar{\alpha}_4 \rho_3 & -\bar{\alpha}_4 \alpha_3 & \ldots\\
  \vdots & \vdots & \vdots & \vdots & \vdots &\ddots\\
\end{array}%
\right)
\end{equation}
where $\rho_k = \sqrt{1 - |\alpha_k|^2}$.
We note that Jacobi matrices, obtained in a similar way for
orthogonal polynomials on the real line, are tridiagonal matrices.
As in the case of orthogonal polynomials on the real line, an
important connection between CMV matrices and monic orthogonal
polynomials is
\begin{equation}\label{q:phidet}
   \Phi_n(z) = \det(z I - \CMV^{(n)})
\end{equation}
where $\CMV^{(n)}$ is the upper left $n \times n$ corner of $\CMV$.

If $|\alpha_{n-1}| = 1$, $\CMV$ decouples between $(n-1)$ and $n$ as
$\CMV = \CMV^{(n)} \oplus \tilde{\CMV}$, where the upper left corner is
an $n \times n$ unitary matrix
$\CMV^{(n)} = \CMV^{(n)}(\alpha_0,\alpha_1,\ldots,\alpha_{n-1})$
and the remaining block $\tilde\CMV(\alpha_n, \alpha_{n+1},\ldots)$
is a (semi-infinite) CMV matrix.
This suggests that a {\em unitary\/} $n\times n$ truncation of a CMV
matrix can be obtained by replacing $\alpha_n^{}\in\Disk$ by
$\beta\in\partial\Disk$.
The truncated CMV matrix has as characteristic polynomial
\begin{equation}\label{q:phibeta}
   \Phi_n(z;\beta) = z\Phi_{n-1}(z) - \beta\Phi_{n-1}^*(z)
\end{equation}
whose zeros are all simple and lie on $\partial\Disk=S^1$, a fact
which will be convenient below.  The polynomials $\Phi_n(z;\beta)$ are
called {\em paraorthogonal\/} polynomials.

Given a set of moments $c_k^{}(\mu)$, we can recover its generating
measure $\mu$ in the classical limit as follows \cite[Thm.~2.2.12]{simon:opuc1}.
Given $\beta$ and $k$, let
$\{z_j\}_{j=1}^k\subset\partial\Disk$ be
the zeros of $\Phi_k(z;\beta)$ and define the atomic measure
\begin{equation}\label{q:mzeta}
   \zeta^{(k)}_\beta := \sum_{j=1}^k \frac{1}{\sum_{i=0}^{k-1}\,|\vfi_i(z_j)|^2}
   \delta_{z_j}^{}
\end{equation}
where $\vfi_i(z)=\Phi_i(z)/\|\Phi_i\|$, and $\Vert \cdot \Vert$
denotes the norm in $L^2(S^1; \mu)$. For any choice of
$\beta_1,\beta_2,\beta_3,\ldots \in\partial\Disk$, $\zeta^{(k)}_{\beta_k}$
converges weakly to $\mu$.  The
weight $\bigl(\sum_i\,|\vfi_i(z_j)|^2\bigr)^{-1}$ is known as the
{\em Christoffel function\/} \cite[p.~117ff]{simon:opuc1}.

% ---------------------------------------------------------------------------

\subsection{OPUCs for SPP}\label{s:opucmv-spp}

Let us now return to McMullen's SPP.  Using the moments
$\mom{k}{\theta}$ computed in Section \ref{s:momcomp}, we used
Gram--Schmidt to construct the orthogonal polynomials
$\Phi_k(z)=\Phi_k(z;\mu_\theta)$ associated to the measure
$\mu_{\theta}$, which for convenience we rotate by $\pi$. Henceforth,
by $\mu_\theta$ we mean this rotated measure. Since all moments are
real and $\mom{3k-1}{\theta}=\mom{3k+1}{\theta}=0$, the polynomials
$\Phi_{3k}(z)$ are, in fact, polynomials in $z^3$ with real
coefficients, and $\Phi_{3k+1}(z) = z \Phi_{3k}(z)$ and
$\Phi_{3k+2}(z) = z^2 \Phi_{3k}(z)$.  It then follows from
\eqref{q:verb0} that the Verblunsky coefficients are all real with
$\alpha_{3k}=\alpha_{3k+1}=0$ and $\alpha_{3k+2}\in(-1,1)$; this also
follows from the symmetries of $\mu$ and
\cite[(1.6.66)]{simon:opuc1}.

It is convenient to introduce monic polynomials
$q_k(z^3):=\Phi_{3k}(z)$ with Verblunsky coefficients $\gamma_k:=\alpha_{3k+2}$,
in terms of which \eqref{q:szego} reads
\begin{equation}\label{q:recq}
   q_{k+1}(z) = zq_k(z) - \gamma_kq_k^*(z).
\end{equation}
In analogy with \eqref{q:phibeta}, we define the paraorthogonal
\begin{equation}
   q_k(z;\beta) = zq_{k-1}(z) - \beta q_{k-1}^*(z),
\end{equation}
where we choose now $\beta=\pm1$ depending on the sign of $\gamma_{k-1}$.
Since $q_k(z^3;\beta)=\Phi_{3k}(z;\beta)$, the zeros of $q_k(z;\beta)$
all lie on $\partial\Disk$.
Moreover, since $q_k$ has real coefficients, its (non-real) zeros occur
in complex conjugate pairs.
For our numerical computations, we used $q_k$ exclusively
in place of $\Phi_{3k}$.
Since the OPUCs depend on the underlying measure $\mu$ only through the
moments $c_j$, it is clear from the definition \eqref{q:cdef} of the latter
that $\{q_k\}_{k=0}^\infty$ are the OPUCs for a measure $\tmu_\theta$ where
\begin{equation} \label{eq:mutilde}
\dtmu_\theta(\phi)=3\,\dmu_\theta(\phi/3)
\end{equation}
for $\phi\in[-\pi/3,\pi/3)$.

% [REVISE after \S\ref{s:opucmv-obs}]
\medskip
Having computed $q_k$ and $\gamma_k=-q_{k+1}(0)$,
the accuracy of the numerical computation can be checked using \eqref{q:szego},
by ensuring that the error (which is zero for exact computation)
\[
  \Err_k:=\tssum_j\,|q_{k+1,j}-q_{k,j-1}+\gamma_k q_{k,k-j}|,
\]
with $q_{k,j}$ the $j$th coefficient of $q_k$ (with $q_{k,-1}=0$),
remains small.
We found that high-precision arithmetic, both for the moments $c_k$
and the subsequent computations involving $q_k$, are crucial
to control the error.
For $\theta=119^\circ$ and $\epsilon=10^{-7}$, computations using
a 256-bit ``quad-double'' precision gives us $k\simeq280$.
% the usual 64-bit floating-point arithmetic gives us $k\simeq40$,
% and 1009-decimal-digit precision gives us $k\simeq480$.

% \bigskip\hbox to\hsize{\qquad\hrulefill\qquad}\medskip

% ---------------------------------------------------------------------------

\subsection{Observations}\label{s:opucmv-obs}

We now present a few numerical observations on the spectral properties
of CMV matrices.
In this section, we work exclusively with the symmetry-reduced
polynomials $q_k$ and the measure $\tmu_\theta$ introduced above.

First, for every $\theta\in(0,2\pi/3)$, the Verblunsky coefficients
are all negative, $\gamma_k<0$.
Seen in the light of the formula for Verblunsky coefficients
for rotated measures \cite[(1.6.66)]{simon:opuc1}, for
any measure $\mu$,
\begin{equation}
   \gamma_k(\mu^{(\alpha)}) = \ex^{-\im\alpha(k+1)}\,\gamma_k(\mu)
\end{equation}
where $\dmu^{(\alpha)}(\phi)=\dmu(\phi-\alpha)$,
our observation means that $\mu_\theta$ belongs to a family of measures
whose Verblunsky coefficients are, possibly after rotation, all negative.
To obtain paraorthogonal polynomials $q_k(z;\beta)$ corresponding
to unitary truncations of the CMV matrix, it is therefore natural to
take $\beta=-1$ (but see the effect of the choice of $\beta$ at the
end of this section).

\begin{figure}[h]
\begin{center}
\psfrag{ 120}[c][c][0.707][0]{$\theta$}
\psfrag{yyy}[r][r][0.707][0]{$\log(1+\gamma_k)$}
\includegraphics[scale=0.40,angle=-90]{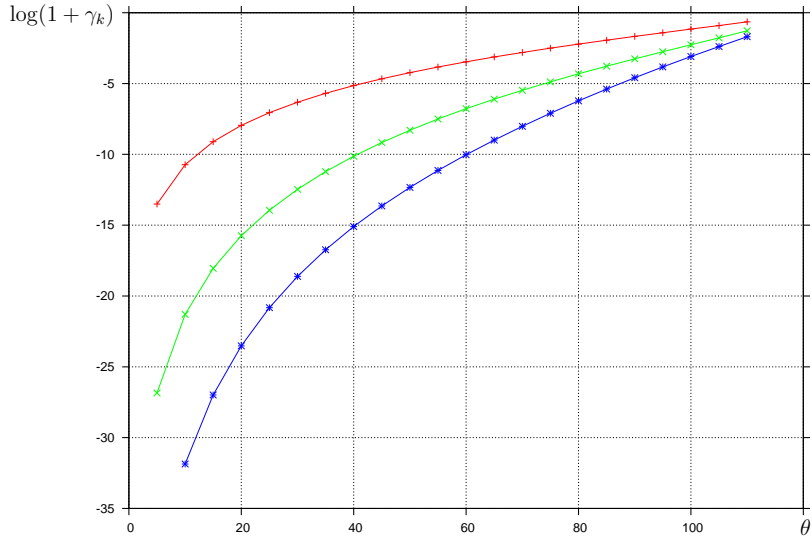}
\caption{Monotonicity of Verblunsky coefficients:
Plot of $\log(1+\gamma_k)$ against $\theta$ for
(from top) $k=1,2,4$.}\label{f:verb}
\end{center}
\end{figure}

Secondly, for every $k$ fixed, $\gamma_k(\mu_\theta)$ increases
monotonically from $-1$ to $0$ as $\theta$ goes from $0$ to $2\pi/3$.
Figure~\ref{f:verb} plots $\log(1+\gamma_k)$ for $k=1$, $2$ and $4$
against $\theta$, showing the rapid convergence of $\gamma_k$ to $-1$
as $\theta\to0$.

% \vfill\eject

In the course of our numerical computations, we found that small
opening angles $\theta$ only allow us to obtain a few polynomials $q_k$
reliably, while larger $\theta$ allows us to obtain more polynomials.
Since the numerical instability of the Gram-Schmidt orthogonalisation
of the polynomials $1,z,z^2,\ldots$ is closely related to the
numerical singularity of the Toeplitz matrices
\begin{equation}
   T^{(n+1)}_\theta = (c_{j-i}^{(\theta)})_{0 \le i,j \le n} =
   ( \langle z^j,z^i \rangle_\theta )_{0 \le i,j \le n},
\end{equation}
we expect that the determinant $D_n(d\mu_\theta) := \det T^{(n+1)}_\theta$
decreases monotonically to $0$ as $\theta \to 0$.
In fact, this would follow from the monotonicity of the individual Verblunsky
coefficients by the identity \cite[\S1.3.2 and (2.1.1)]{simon:opuc1}
\begin{equation}
   D_n(d\mu_\theta) = \prod_{j=0}^{n-1}\,(1-|\gamma_j(\mu_\theta)|^2)^{n-j}.
\end{equation}

% \vfill\eject

\begin{figure}[h]
\begin{center}
\psfrag{phi}[c][c][0.707][0]{$\phi$}
\psfrag{pi}[c][c][0.707][0]{$\pi$}
\psfrag{ 1}[r][r][0.707][0]{$1$}
\includegraphics[scale=0.45,angle=-90]{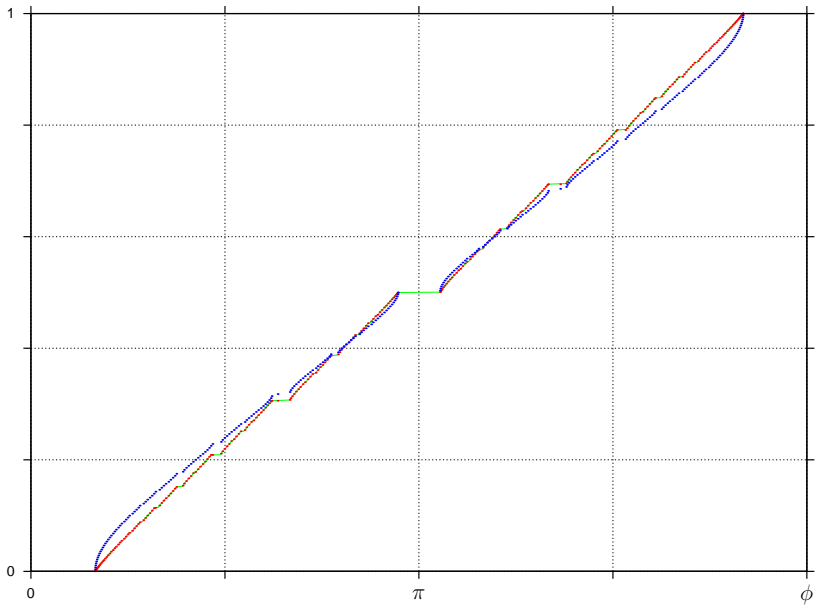}
\caption{Solid (green) line: (classical) measure $\mu_{119}^{}$,
red dots (overlapping the line): $\zeta^{(280)}$,
blue dots: $\mathrm{idz}^{(280)}$.}\label{f:z119}
\end{center}
\end{figure}

\medskip
In Figure~\ref{f:z119}, we present our numerical computations for
$\theta=119^\circ$ with $\epsilon=10^{-7}$ (size of the largest cell).
The thin (green) line is the distribution function of the
classical measure $\tmu_{119}([0,\phi])$.
As noted earlier, $\dtmu_\theta$ tends to the Lebesgue measure
$\dphi/(2\pi)$ as $\theta\to120^\circ$, but significant ``gaps''
(i.e.\ intervals $I$ with $\tmu(I)=0$) remain even for $\theta=119^\circ$.
Let $z_j$ be the zeros of the paraorthogonal $q_{280}^{}(z;\beta\!=\!-1)$,
ordered from $\phi=0$ to $2\pi$.
The blue dots in Fig.~\ref{f:z119} plot the ``integrated density of zeros''
measure
\begin{equation}
   \mathrm{idz}^{(k)} = \sum\nolimits_{j=1}^k\,\frac{1}{k} \delta_{z_j}^{},
\end{equation}
while the red dots plot the ``Christoffel-corrected'' measure $\zeta^{(280)}$
defined in~\eqref{q:mzeta}.
One can see that $\zeta^{(280)}$ overlaps the classical measure $\tmu_{119}$
to within plotting line thickness.
This agreement confirms the accuracy of our numerical computations.
We note that one may find zeros in the spectral gaps
$S^1\backslash\mathrm{supp}\,\mu_\theta$, but with smaller
probability as $k\to\infty$ and the Christoffel function
at these ``spurious'' zeros is small.

The idz plot also suggests that in the limit $k\to\infty$, zeros appear
to accumulate at the gap edges, most visibly for the main gap centred
at $0=2\pi$ and the secondary gap starting at $\pi$.
As noted above, there could be zeros inside the gaps;
here they are most visible near $2\pi/3$ and $4\pi/3$
[see also Fig.~\ref{f:beta} below].
However, their contribution to the measure $\zeta^{(k)}$ is weighted
down by the Christoffel function.

% \vfill\eject

\begin{figure}[h]
\begin{center}
\psfrag{arg(zk)}[c][c][0.707][0]{arg$\,z_k$}
\psfrag{pi}[c][c][0.707][0]{$\pi$}
\psfrag{arg(B)}[r][r][0.707][0]{arg$\,\beta$}
\includegraphics[scale=0.45,angle=-90]{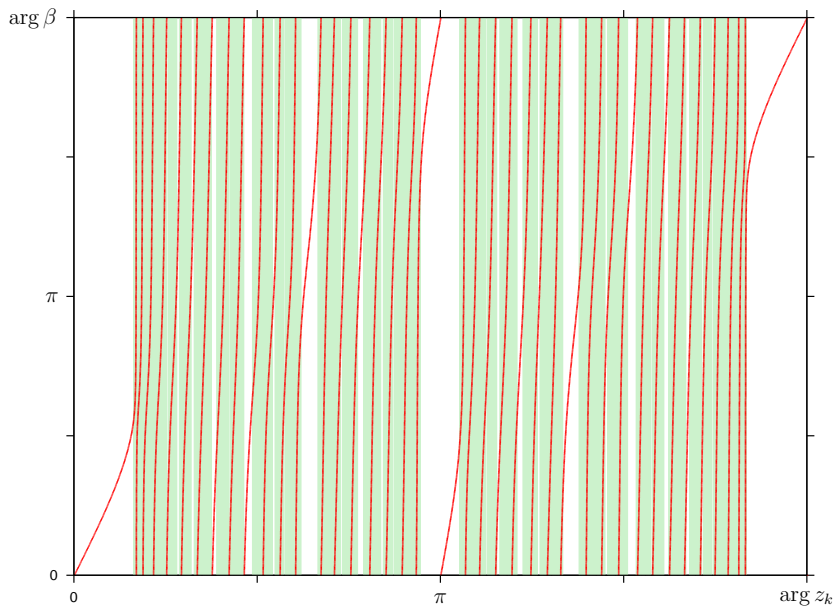}
\caption{
Migration of zeros of $q_{40}(z;\beta)$ as a function of $\beta$ for
$\theta=119^\circ$.
Horizontal axis is arg$\,z_k$ and the vertical axis is arg$\,\beta$.
Each vertical ``line'' is in fact made up of 500 disjoint dots.
Faint (green) background is the approximate support of the classical
measure $\mu_{119}^{}$.
}\label{f:beta}
\end{center}
\end{figure}

In Figure~\ref{f:beta}, we plot,
for $\theta=119^\circ$ and $\beta\in S^1$,
the migration of the zeros of $q_{40}(z;\beta)$
as arg$\,\beta$ varies from $0$ to $2\pi$,
superimposed on the support of the classical measure $\mu_{119}^{}$
(since supp$\,\mu_{119}^{}$ has Lebesgue measure zero,
the ``solid'' background arises only from plotting limitations).
This plot confirms Thm.~1.1 in \cite{simon:07b}, which states that
any interval $[a,b]\not\in\textrm{supp}\,\mu$, represented by a white
background, contains at most one zero of $q(z;\mu,\beta)$.
It also confirms Thm.~1.3 in \cite{simon:07b}, stating that for any
$\beta$, $\beta'\in S^1$, the zeros of $q_k(z;\beta)$ and of
$q_k(z;\beta')$ strictly interlace. The graph also shows that
each zero $z_k^{(\beta)}$ is monotone in arg$\,\beta$ and travels from
$z_k^{(0)}$ to $z_{k+1}^{(0)}$ as arg$\,\beta$ increases from $0$ to $2\pi$.

% \bigskip\hbox to\hsize{\qquad\hrulefill\qquad}\medskip

% ===========================================================================

\nocite{gittins-npi-stoiciu-dw:opuc}

% The following refs are not yet cited in the text -- pls remove when done:

% \nocite{avila-last-simon:10}
% \nocite{killip-stoiciu:09}
% \nocite{kiselev-last-simon:97}
% \nocite{kotani-ushiroya:88}
% \nocite{last-simon:08}
% \nocite{lubinsky:09}
% \nocite{mehta:rm}
% \nocite{press-al:nr3}

\end{document}